\documentclass[11pt]{article}
\usepackage[margin=1in]{geometry}

\usepackage[utf8]{inputenc}
\usepackage{lmodern}
\usepackage[normalem]{ulem} 
\usepackage{float}
\usepackage{enumitem}
\usepackage[utf8]{inputenc} 
\usepackage{tocloft}
\usepackage{graphicx, titlesec}
\graphicspath{/figures}
\usepackage{amscd,amsmath,amsthm,amssymb}
\usepackage{verbatim}
\usepackage[all]{xy}
\usepackage{color}
\usepackage{url}
\usepackage{tikz}
\usepackage{mdframed, xcolor}
\usepackage[]{hyperref}

\hypersetup{
    colorlinks=true,
    linkcolor=violet,
    filecolor=black,      
    urlcolor=violet,
    citecolor=violet,
    pdfpagemode=FullScreen,
    }

\newtheorem{theorem}{Theorem}
\newtheorem{problem}[theorem]{Problem}
\newtheorem{lemma}[theorem]{Lemma}
\newtheorem{conj}[theorem]{Conjecture}
\newtheorem{proposition}[theorem]{Proposition}
\newtheorem{corollary}[theorem]{Corollary}
\theoremstyle{definition}
\newtheorem{remark}[theorem]{Remark}
\newtheorem{example}[theorem]{Example}
\newtheorem{definition}[theorem]{Definition}

\newcommand{\sgn}{\operatorname{sgn}}
\newcommand{\diag}{\operatorname{diag}}
\newcommand{\rank}{\operatorname{rank}}

\newcommand{\M}{\mathcal{M}}
\newcommand{\RR}{\mathbb{R}}
\newcommand{\CC}{\mathbb{C}}
\newcommand{\Jac}{\text{Jac}}

\definecolor{bl}{RGB}{10,10,100}

\title{Constraining the outputs of ReLU neural networks} 
\author{Yulia Alexandr$^{1}$ and Guido Mont\'{u}far$^{1,2}$}
\date{$^{1}$University of California, Los Angeles\\
$^{2}$Max Planck Institute for Mathematics in the Sciences} 

\begin{document}

\maketitle

\begin{abstract}

We introduce a class of algebraic varieties naturally associated with ReLU neural networks, arising from the piecewise linear structure of their outputs across activation regions in input space, and the piecewise multilinear structure in parameter space. By analyzing the rank constraints on the network outputs within each activation region, we derive polynomial equations that characterize the functions representable by the network. We further investigate conditions under which these varieties attain their expected dimension, providing insight into the expressive and structural properties of ReLU networks. 
\end{abstract}

\section{Introduction}

Rectified linear unit (ReLU) neural networks are foundational to modern deep learning, underpinning advances in applications ranging from image recognition and natural language processing to complex decision-making systems. Their effectiveness stems in part from the simplicity and computational efficiency of the ReLU activation function. Although the ReLU activation function is not everywhere differentiable and can lead to dead neurons, its positive linear behavior has been found to help mitigate vanishing gradients and facilitate the training of deep architectures. Despite their empirical success, a complete theoretical understanding of the functions these networks compute---and the structural constraints on their outputs---remains elusive.

In this work, we develop an algebro-geometric framework for analyzing the outputs of ReLU networks. Our approach focuses on characterizing polynomial constraints that govern the network outputs across different input points when the parameters vary within activation regions where the parametrization is smooth. We associate algebraic varieties to these regions and derive defining equations from rank constraints on the network outputs. This perspective reveals structural limitations on the functions represented by the network and provides insight into its expressive capacity. We analyze these varieties across a range of architectures, including shallow and deep fully connected networks, with and without biases, and examine both input data generating a single pattern of neuron activations and multiple patterns of neuron activations. 

Figure~\ref{fig:functions-parameters-input} illustrates the outputs of a ReLU network as a function of the network parameters for a fixed input dataset, and as a function of the input data for fixed network parameters. In both cases the function is piecewise smooth, dividing the parameter space and the input space, respectively, into activation regions where the function is smooth. The geometric and combinatorial properties of these subdivisions are important in the theoretical analysis of ReLU networks and have been investigated intensively over the years. 
In particular, the subdivisions of the input space have been used to reason about the impact of the architecture choice (e.g., deep or shallow) on the expressive power and approximation errors. The subdivisions of the parameter space have been used in the evaluation of complexity measures in classical statistical learning theory. They have also been used to study parameter optimization through the Gram matrix of the Jacobian of the output features, both for networks in kernel regimes and networks in active feature learning regimes. 

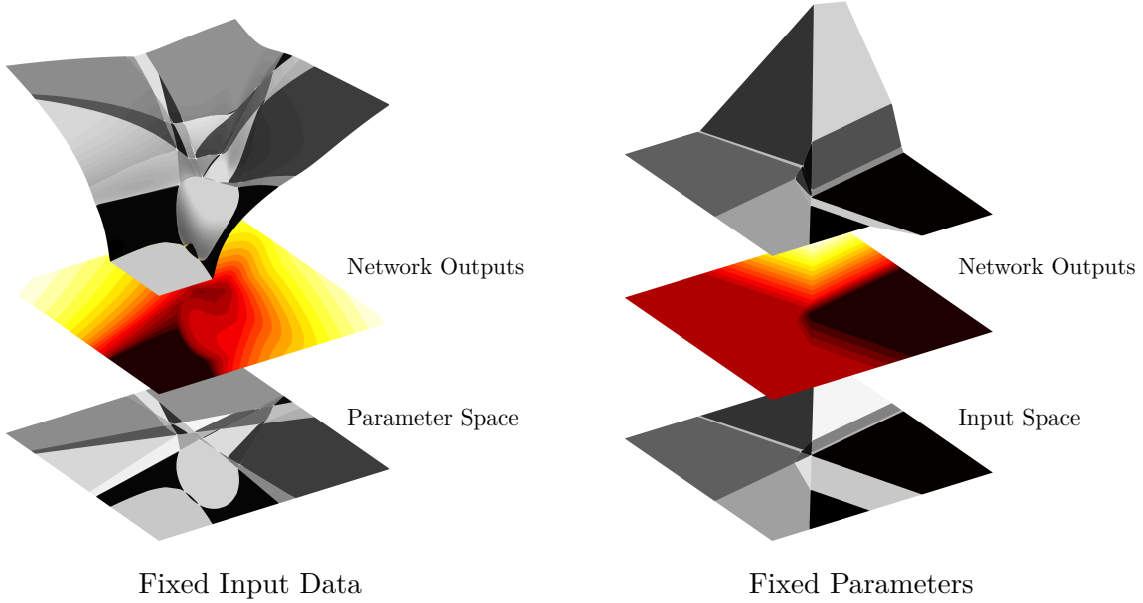
\begin{figure}[h]
\centering
\begin{tabular}{cc}
\scalebox{.8}{ 
\begin{tikzpicture}
\node[] at (0,0) {\includegraphics[clip=true, trim=3cm 13cm 3cm 12cm, width=7cm]{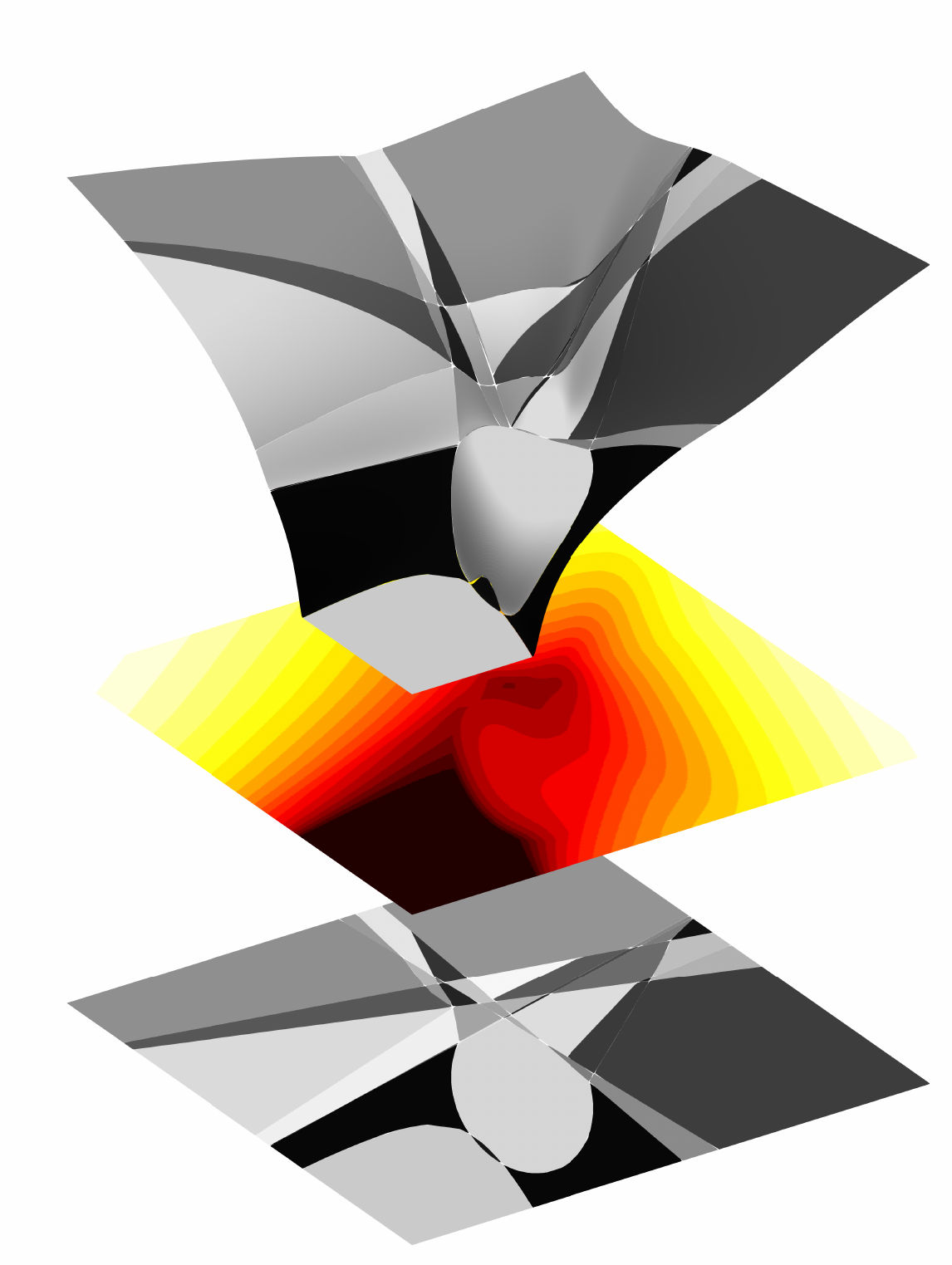} };
\node[anchor=west] at (2.5,0) {Network Outputs}; 
\node[anchor=west] at (2.5,-2.5) {Parameter Space}; 
\end{tikzpicture}
}
&
\scalebox{.8}{ 
\begin{tikzpicture}
\node[] at (0,0) {\includegraphics[clip=true, trim=3cm 13cm 3cm 12cm, width=7cm]{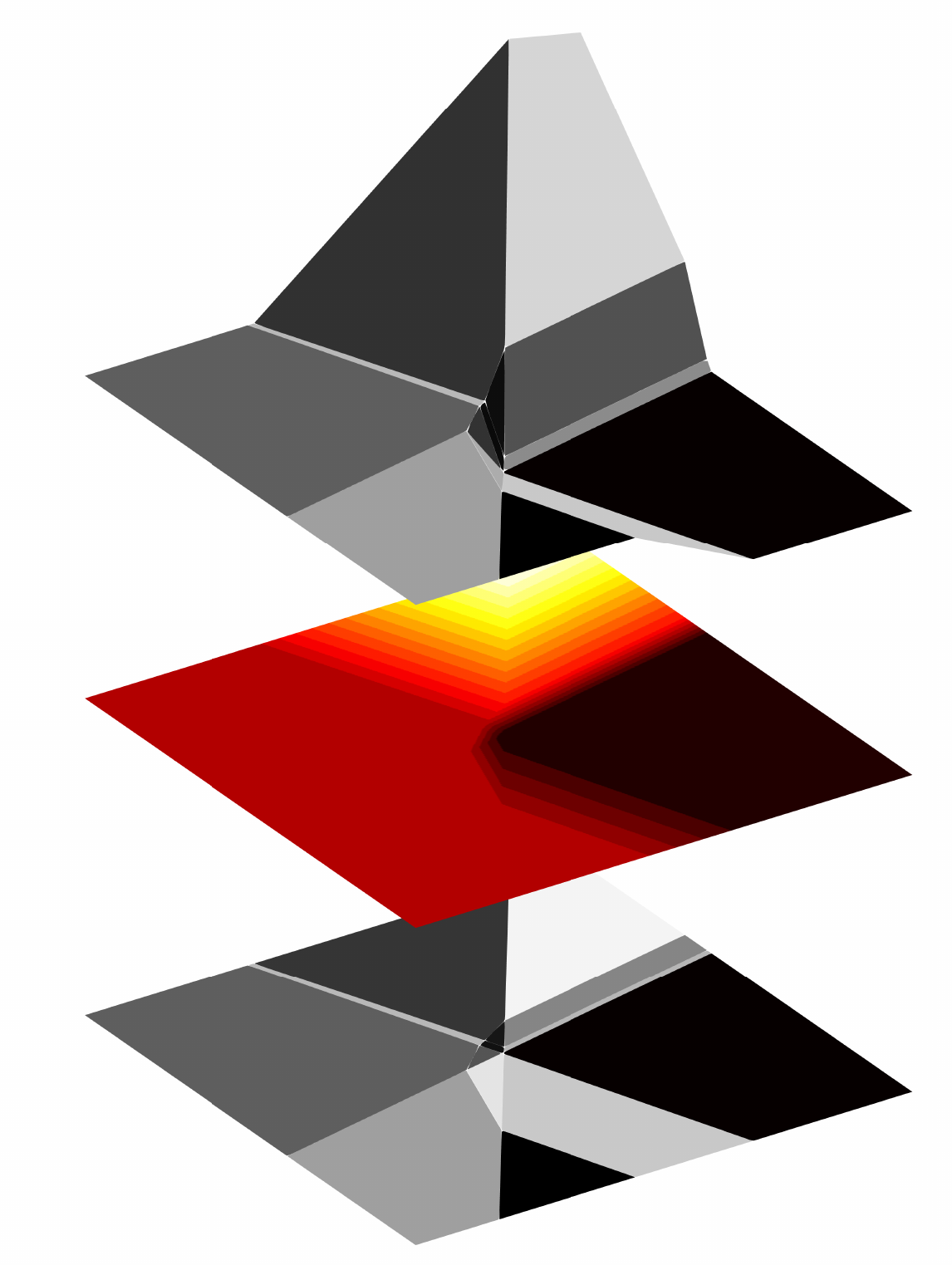} };
\node[anchor=west] at (2.5,0) {Network Outputs}; 
\node[anchor=west] at (2.5,-2.5) {Input Space}; 
\end{tikzpicture}
}
\\
Fixed Input Data & Fixed Parameters
\end{tabular} 
\caption{
For a three-layer ReLU network, the left panel fixes a dataset $X$ and visualizes the network outputs (the log-one-plus of the norm of the array of outputs) as a function of the parameter $\theta$ over a 2D slice of the parameter space. 
The right panel fixes the value of the parameter $\theta$ and visualizes the network outputs (the sum of the output coordinates) as a function of the input data $x$ over a 2D slice of the input space. 
The top and middle visualize the output as a graph and as a heatmap, respectively. 
The bottom and top are colored by activation regions, which correspond to the activation patterns of all ReLUs in the network at the particular input data and parameter. 
The output is piecewise polynomial as a function of the parameter and piecewise linear as a function of the input data. 
Here we used a network with input dimension $n_0=2$, two ReLU hidden layers of widths $n_1=3$, $n_2=2$, and a linear output layer of width $n_3=2$, and a dataset $X$ of size $m=2$. 
} 
\label{fig:functions-parameters-input}
\end{figure}

One observation is that ReLU neural networks are capable of representing functions that have an enormous number of linear regions over the input space. The number of linear regions can be exponential in the number of network parameters.
This means that even if the functions may have a very complex appearance, they must also have many regularities. 
Specifically, since the linear pieces of the represented functions share the same set of parameters, they cannot be independent of each other. 
Such dependencies are important because they directly affect the ability of a network to generalize outside of a training dataset. Studying these invariants can also provide insight for neural network verification, by imposing additional algebraic constraints on the output beyond the standard convex set or bounding box approaches.
Despite the intuitive appeal, characterizing the specific structure of the dependencies has remained an open problem. 
In this work we take systematic steps towards describing those dependencies.

\subsection{Contributions} 

This work develops an algebraic framework for analyzing the sets of possible outputs produced by ReLU neural networks, based on the polynomial constraints satisfied by their outputs. 
Our main contributions are as follows. Except for the preliminaries presented in Section~\ref{sec:preliminaries}, the results presented in this chapter are new. 

\begin{itemize}
    \item In Section \ref{sec:preliminaries}, we recall the piecewise multilinear structure of ReLU networks and establish that, on each activation region, the output map is multilinear in the network parameters with degree equal to the number of layers (Proposition \ref{prop:multilinear-structure}). We also recall the explicit pathwise parameterization of the output matrix in terms of active paths (Proposition \ref{prop:comb-paths}).

    \item In Section \ref{sec:relu-varieties}, we define the \textit{ReLU output variety} as the Zariski closure of the image of the output map for a fixed activation pattern and dataset, and formulate the implicitization problem of determining its defining equations. 
    We also introduce the \textit{ReLU pattern variety}, which captures the parameterized outputs given a fixed activation pattern, independently of any specific dataset. 

    \item In Section \ref{sec:single-block}, we analyze single-block ReLU output varieties, where all input data points lie in the same activation region. 
    We show that they are linear sections of determinantal varieties. 
    We give explicit generators of the defining ideals for both networks without biases (Proposition \ref{prop:single-block-gens-linear}) and networks with biases (Proposition \ref{prop:single-block-with-biases}). 

    \item  In Section \ref{sec:from-fun-to-pred}, we show how invariants of the ReLU pattern variety can be translated into invariants of the ReLU output variety. 

    \item In Section \ref{sec:two-blocks}, we focus on the defining ideals of two-block ReLU pattern varieties. 
    We identify mixed determinantal constraints involving concatenations and differences of the output matrices for shallow and deep networks (Theorems \ref{thm:shallow-invariants} and \ref{thm:deep-invariants}). 
    We conjecture that there are no other invariants in the shallow case. 

    \item In Section \ref{sec:mult-blocks}, we extend our analysis to ReLU pattern varieties with multiple activation blocks and exhibit a class of determinantal invariants that constrain their image. 

    \item Finally, in Section \ref{sec:dim-shallow}, we study the dimensions of ReLU pattern varieties for shallow networks and provide a sufficient condition on the layer widths under which the expected dimension is achieved (Theorems \ref{thm:expected-dim} and \ref{thm:expected-dim-multiple-blocks}). 
\end{itemize}

\subsection{Combinatorial and implicit approaches to deep learning} 

Our work contributes to a broader program developing combinatorial and implicit approaches to deep learning. 
These consider, for example, the subdivisions of the parameter space of a neural network into activation regions, or the set of parameters for which the network interpolates a given dataset, or the implicit description of the possible outputs of a neural network as the set of solutions to a list of equations and inequalities. 
We give just a brief overview of some of the related works in this context. 

The works \cite{NIPS2014_109d2dd3, pascanu2014number} proposed to study the combinatorial complexity of the functions that can be represented by ReLU networks as a way to distinguish between deep and shallow architectures. They showed that for a fixed number of neurons or a fixed number of parameters, deep architectures can represent functions that have many more linear pieces than any of the functions that can be represented by shallow architectures. 
Linear regions also appeared in the work \cite{pmlr-v49-telgarsky16} discussing benefits of depth and several works have worked on establishing bounds \cite{noteson, pmlr-v80-serra18b}. 
Interestingly, not only is any function represented by a ReLU network continuous and piecewise linear, but also the converse is true. As shown in \cite{arora2018understanding}, any continuous piecewise linear function can be represented by a ReLU network, provided the architecture is deep enough. 
Establishing the minimum depth, depending on the number of input dimensions, that is sufficient for an infinitely wide ReLU network to exactly represent any continuous piecewise linear function is an interesting problem that has been studied in \cite{haase2023lower, NEURIPS2021_1b9812b9}. 

The works \cite{charisopoulos2019tropical, zhang2018tropical} proposed to regard neural networks with continuous piecewise linear activation functions as parametric tropical rational functions. A tropical rational function is a difference of two convex piecewise linear functions. 
This offers a geometric interpretation of networks with piecewise linear activation functions in terms of polytopes. 
Based on these notions, \cite{doi:10.1137/21M1413699} obtained explicit formulas for the number of linear regions of the functions represented by shallow maxout networks and sharp asymptotic bounds for deep maxout networks. 
More precisely, any continuous piecewise linear function can be decomposed as a difference of two convex piecewise linear functions. 
Any convex piecewise linear function can be described in terms of its convex conjugate, which encodes the gradients and intercepts of the linear pieces, and this can be described in terms of a polytope, called the Newton polytope. 
Taking this perspective, the output function at any neuron of the network corresponds to a pair of polytopes, and composition with a layer corresponds to taking Minkowski sums and convex hulls of such polytopes. 
The vertices of the polytope stand in correspondence with the linear regions and Minkowski sums correspond to taking superpositions or arrangements of hypersurfaces separating linear regions. 
Based on this relationship, \cite{escobar2023enumeration} obtained formulas for the number of linear regions of max-pooling layers by equivalently enumerating the vertices of certain Minkowski sums of simplices. 

An interesting aspect are the links that can be established between combinatorial properties of the represented functions and properties of the network parameters. 
In particular, \cite{NEURIPS2019_9766527f} showed that for natural classes of probability distributions of parameters, the expected value of the number of linear regions of the represented functions can be much smaller than the theoretical maximum. 
This result was obtained by analyzing the volume of the non-linear locus using a change of variables approach that incorporates the distribution of parameters. 
The approach was later also applied to study maxout networks in \cite{NEURIPS2021_f2c3b258}. 
The distribution of the non-linear locus can also be related to the curvature of the function and this manifests in the implicit bias of gradient descent. For example, \cite{JMLR:v24:21-0832,liang2025implicit} showed for overparametrized ReLU networks in the lazy regime that gradient descent returns a solution function that interpolates the training data and minimizes a curvature penalty depending on the distribution of the non-linear locus at initialization. 
Moreover, the distribution of linear regions can be linked to emerging phenomena in learning, such as grokking~\cite{humayun2024deep}. The recent work \cite{patel2025on} relates the local complexity of a ReLU network, a measure of the density of the non-linear locus over the input space near the input data distribution, and notions such as adversarial robustness and feature learning.

Another question of interest when investigating the set of functions represented by a neural network is the dimension, which has implications to optimization, parameter recovery, and training data recovery. The dimension can be defined in multiple ways, each capturing different facets of the model’s expressive power. 
One natural approach is to fix a finite dataset and consider the set of all prediction vectors the network can produce on that dataset as the parameters vary. The dimension is then defined as the dimension of this set in the ambient Euclidean space. This perspective highlights the degrees of freedom available to the network when fitting or classifying the given data and provides a concrete measure that is sensitive to both architecture and data geometry. 
The work \cite{grigsby2022functional} investigates this notion of functional dimension and how parameter symmetries cause variation in expressivity across parameter space. 
The work \cite{phuong2020functional} studied functional vs parametric equivalence of ReLU networks. Interestingly, as shown in \cite{pmlr-v119-rolnick20a}, it is often possible to recover the parameters and network architecture from the outputs of a ReLU network. The identifiability of ReLU networks has received attention, for instance, in \cite{10.1007/s10994-023-06355-4, bona-pellissier2022local, VLACIC2021107485, embedding-ReLU-identifiability}. 

We observe that the functional dimension is characterized as the maximum rank of the Jacobian of the network’s output with respect to its parameters at smooth points. This Jacobian rank coincides with the rank of the empirical Neural Tangent Kernel (NTK) Gram matrix. 
The NTK \cite{jacot2018neural} has been studied intensively in the context of parameter optimization, memorization, and generalization. 
In particular, it is known that if a network is sufficiently overparametrized, then the NTK is positive definite with high probability over the parameter initialization. A series of works have investigated the degree of overparametrization that is sufficient to obtain a well-conditioned NTK, e.g., \cite{banerjee2023neural,bombari2022memorization,montanari2022interpolation,nguyen2021tight}. The recent work \cite{karhadkar2024bounds} obtained bounds on the minimum eigenvalue of the NTK for ReLU networks in terms of a notion of separation between the input data points.

Combinatorial and implicit approaches can offer insights into how the degree of overparametrization and structure of the data impact the optimization problem. Restricted to a finite data set, the parametrization map of a ReLU network is multilinear. This has implications to the optimization problem, some of which have been discussed by \cite{pmlr-v80-laurent18b} for the case of optimization with the hinge loss. When the Jacobian is full rank, the gradient of a loss can only vanish if the gradient in function space vanishes, which for many typical losses can only happen when the model interpolates the training data. This perspective has been studied for ReLU networks in \cite{karhadkar2024mildly}. 
The investigation of the rank of the Jacobian translates to combinatorial problems over regions of the parameter space. 
The work~\cite{matena2022a} uses a combinatorial approach to optimize ReLU networks by leveraging their partitioning of input space into linear regions, allowing exact and efficient training on small datasets through simpler linear problems with provable convergence. 
The work \cite{brandenburg2024the} studies the structure of the loss surface of a ReLU network for binary classification from a combinatorial perspective.

In regard to the implicit description of function spaces, we may highlight works pursuing a similar program for the case of linear convolutional networks \cite{kohn2022geometry,doi:10.1137/23M1565504} and for polynomial networks~\cite{arjevani2025geometry}. 
A description of the set of outputs of a shallow ReLU network for a finite dataset in one dimension appeared in \cite{karhadkar2024mildly}. 
Moreover, \cite[Theorem 3.2]{brandenburg2024the} obtained a result describing the set of functions that can be represented by a deep ReLU network as a (semi-algebraic) subset of the parameters of a standard representation of continuous piecewise linear functions.

\section{Preliminaries on ReLU networks}\label{sec:preliminaries}
Consider a \textit{ReLU network}, which gives rise to a function taking inputs and parameters to outputs,
$f\colon \mathbb{R}^{n_0} \times \mathbb{R}^{d_{\operatorname{par}}} \to \mathbb{R}^{n_L}$,  
\begin{equation}\label{eq:f-theta-x}
f_\theta(x) = g_L \circ \sigma \circ g_{L-1} \cdots \sigma \circ g_1 (x). 
\end{equation}
Each layer of the network $\ell =1,\ldots, L-1$ consists of a parametric affine map, 
$$
g_\ell\colon \mathbb{R}^{n_{\ell-1}}\to \mathbb{R}^{n_\ell}; \; y \mapsto W^{(\ell)} y + b^{(\ell)}, 
$$ 
with parameter $\theta_\ell=[W^{(\ell)}, b^{(\ell)}]$, followed by the ReLU activation function $\sigma$ applied component-wise, 
$$
\sigma \colon \mathbb{R}^{n_\ell} \to\mathbb{R}^{n_\ell};\; y\mapsto [\max\{0,y_1\} , \ldots, \max\{0,y_{n_\ell}\}]. 
$$
The \textit{architecture} of the network is determined by the sequence of layer widths $n_0,n_1,\ldots, n_L$, and the map $f_\theta(x)$ can be viewed as a composition map $\RR^{n_0}\to\RR^{n_1}\to\cdots\to\RR^{n_L}$ from the \textit{input layer} to the \textit{output layer}. We call all the $L-1$ layers in-between \textit{hidden layers}.

The collection of all parameters of the network is denoted $\theta =[W^{(1)}, b^{(1)},\ldots, W^{(L)},b^{(L)}]\in \mathbb{R}^{d_{\operatorname{par}}}$, where $d_{\operatorname{par}}=\sum_{\ell=1}^L n_\ell(n_{\ell-1}+1)$. 
We assume that the output layer does not include an activation function but only an affine map, as seen in (\ref{eq:f-theta-x}). Each output coordinate of each layer is called a neuron or a unit. 
We will write $N=n_1+\cdots+n_{L-1}$ for the number of hidden units. For each $i\in\{1,\ldots, N\}$ we let $h_i(x)$ denote the post-activation value of the $i$th hidden unit, which is the output value at the $i$th hidden unit after $\sigma$ has been applied.

\subsection{Activation regions} The function $f_\theta(x)$ can be viewed as a function in the data $x$ for a chosen parameter $\theta$, or alternatively as a function in the parameter $\theta$ for fixed input data $x$. Both perspectives are meaningful and lead to the natural subdivisions of the input space and the parameter space, respectively. 

\paragraph{Subdivision of input space.} 
The function $f_\theta$ is a continuous piecewise linear function of the input data point $x$. For any choice of the parameter $\theta$, the input space $\mathbb{R}^{n_0}$ is subdivided into regions where $f_\theta$ is a linear function of $x$. 

The linear regions are characterized by the \textit{activation patterns} $A(x) = [A_1(x),\ldots, A_N(x)] \in\{0, 1\}^{N}$, which record which units of the network are active (1) or not (0) at any input $x$. 
All input points $x\in \RR^{n_0}$ that follow the same activation pattern $A\in \{0, 1\}^N$ define an activation region in the input space: 
\begin{equation*}
R^A = \{x \in \mathbb{R}^{n_0}\colon \sgn h_i(x) = A_i, i=1,\ldots, N \} , 
\end{equation*}
where $\sgn(x)=1$ if $x>0$ and $\sgn(0)=0$. 
For each such $A$, the function $f_\theta$ restricted to $R^A$ is a linear function of $x$. Note that for any $\theta$, each activation region of $f_\theta$ is contained in a linear region. However, for some choices of $\theta$, a linear region may comprise more than one activation region. 
We also observe that each activation region $R^A$ is determined by finitely many linear inequalities, forming a polyhedron in $\mathbb{R}^{n_0}$. 
It may be empty if there are no inputs for which the activation pattern is $A$. 
The collection of all these regions partitions the input space, i.e., the input space is a disjoint union of the activation regions, $\mathbb{R}^{n_0} = \sqcup_A R^A$.\footnote{Some works define activation regions based on the sign of the pre-activations, then taking values $\{-1,0,+1\}$. 
For our discussion it will be sufficient to consider the binary sign $\{0,1\}$ of the post-activation values. Note that we define the activation regions using a mix of strict and non-strict inequalities so that different activation regions are~disjoint.}

\paragraph{Subdivision of parameter space.} 
The subdivision of the input space naturally gives rise to a subdivision of the parameter space when considering multiple input points. 
Consider a finite data set $X = [x^{(1)},\ldots, x^{(m)}] \in\mathbb{R}^{n_0\times m}$, and denote the map taking parameter values to arrays of output values over this dataset by 
$$
F_X(\theta) = [f_\theta(x^{(1)}), \ldots, f_\theta(x^{(m)})] \in\mathbb{R}^{n_L\times m}. 
$$
If the network has a single output coordinate, then $F_X(\theta)$ is a vector in $\mathbb{R}^m$. For a fixed~$X$, consider a tuple $\mathbf{A}(X) = [A(x^{(1)}),\ldots, A(x^{(m)})]$ with $A(x^{(j)})\in\{0,1\}^{N}$ recording the activation pattern of the network for each of the input data points $x^{(j)}$. 
For any $\mathbf{A} \in\{0,1\}^{N\times m}$ we define a corresponding \textit{activation region} in parameter space as 
\begin{equation*}
S^{\mathbf{A}}_X = \{\theta\in \mathbb{R}^{d_{\operatorname{par}}} \colon \sgn h_i(x^{(j)}) = \mathbf{A}_{ij}, i=1,\ldots, N, j=1,\ldots, m\} . 
\end{equation*}
For each such $\mathbf{A}$, the function $F_X$ restricted to $S^{\mathbf{A}}_X$ is a multilinear function of $\theta$. 
As before, some of the activation regions may be empty. The parameter space is subdivided as $\mathbb{R}^{d_{\operatorname{par}}} = \sqcup_A S^{\mathbf{A}}_X$.

In the case when $A(x^{(i)}) = A(x^{(j)})$ for all $i\neq j$ and all entries in $A(x^{(i)})$ are 1, all neurons of the network are active on each $x^{(i)}$ and the activation function acts as an identity at each layer. In this case we may regard the network as a \textit{fully connected linear} network over the input data $X$. We will write $G_X(\theta)$ to denote $F_X(\theta)$ in this case. 

In general, the different input data points in $X$ may lie in different activation regions in the input space. 
The main objective of this paper is to describe the polynomial constraints defining the image of $F_X(\theta)$ for a fixed but arbitrary dataset $X$ as the parameter $\theta$ varies over an arbitrary activation region $S_X^{\mathbf{A}}$.

\begin{example}\label{ex:different-activations}
Consider the neural network with $L = 2$, $n_0 = n_1 = n_2 = 2$. For a fixed parameter $\theta = [W^{(1)}, W^{(2)}]$, any data point $x$ may follow one of the four activation patterns $[0,0], [0,1], [1,0]$, or $[1,1]$. In Section \ref{sec:relu-varieties}, we will introduce \textit{pattern varieties}. In this case, the pattern variety associated to an activation pattern $A$ is parametrized as $[W^{(1)}, W^{(2)}]\mapsto M(\theta)$ where $M(\theta) = W^{(2)}\diag(A)W^{(1)}$. Let $\M_{A}$ denote the Zariski closure of this map. We plot $\M_{[0,0]}, \M_{[0,1]}, \M_{[1,0]}, \M_{[1,1]}$ from left to right in Figure \ref{fig:independence}, all intersected with $\mathbb{R}_{\geq 0}^{2\times 2}$ at the affine hyperplane where all coordinates of the image sum to 1. Statisticians will recognize $\M_{[0,1]}$ and $\M_{[1,0]}$ as the model of two binary independent random variables, corresponding to the variety of $2\times 2$ minors of rank at most 1. Note that the parametrizations among these four linear pieces are \textit{not} independent; for example, $\M_{[1,0]}$ and $\M_{[1,1]}$ share the parameters $w^{(1)}_{11}, w^{(1)}_{12}, w^{(2)}_{21}, w^{(2)}_{11}$. Our goal is to study the algebraic interactions between these linear~pieces.

\begin{figure}[H]\label{fig:independence}
\centering
\begin{tikzpicture}[scale=1.5, every node/.style={transform shape}]
\tikzstyle{neuron}=[circle, line width=1pt, draw=bl!70, inner sep=.01cm, minimum size = .55cm, 
fill=bl!10
]

\node[neuron, fill=bl!10] (V1) {};
\node[neuron, fill=bl!10] (V2) [below of = V1] {};
\node[neuron] (V3) [node distance=1.5cm, right of = V1] {};
\node[neuron] (V4) [below of = V3] {};
\node[neuron] (V5) [node distance=1.5cm, right of = V3] {};
\node[neuron] (V6) [below of = V5] {};

\draw[->, line width = .8pt] (V1) to node [above] {\tiny$w^{(1)}_{11}$} (V3);
\draw[->, line width = .8pt] (V1) to node [above] {\tiny$w^{(1)}_{21}$} (V4);
\draw[->, line width = .8pt] (V2) to node [below] {\tiny$w^{(1)}_{12}$} (V3);
\draw[->, line width = .8pt] (V2) to node [below] {\tiny$w^{(1)}_{22}$} (V4);

\draw[->, line width = .8pt] (V3) to node [above] {\tiny$w^{(2)}_{11}$} (V5);
\draw[->, line width = .8pt] (V3) to node [above] {\tiny$w^{(2)}_{21}$} (V6);
\draw[->, line width = .8pt] (V4) to node [below] {\tiny$w^{(2)}_{12}$} (V5);
\draw[->, line width = .8pt] (V4) to node [below] {\tiny$w^{(2)}_{22}$} (V6);

\end{tikzpicture}

\centering 
\includegraphics[clip=true, trim=4.5cm 5cm 3.5cm 2.5cm, width=\textwidth]{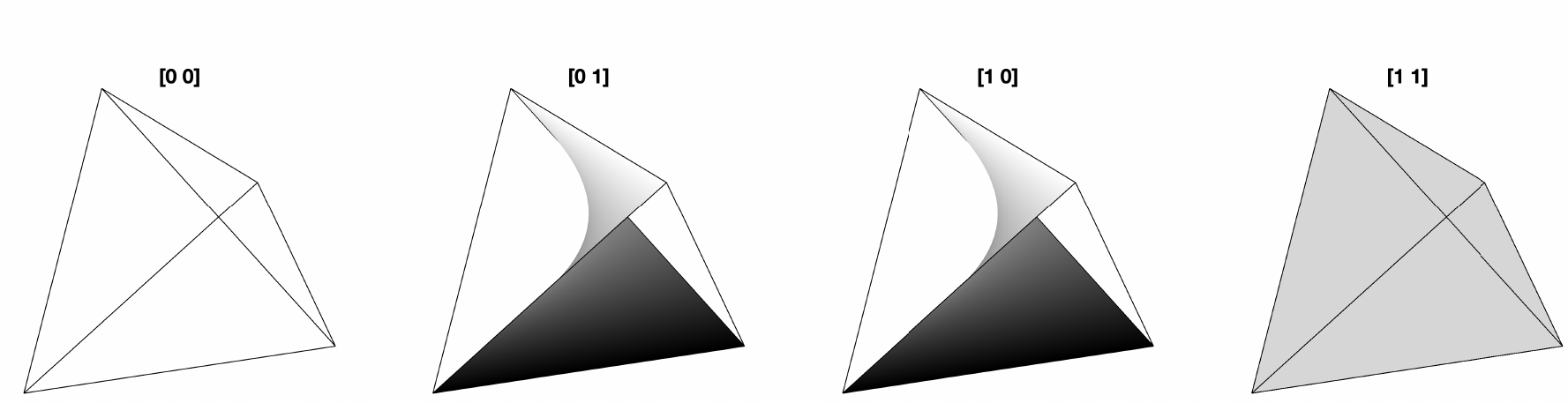}
\caption{Visualization of the four pattern varieties from Example~\ref{ex:different-activations}.  
Each panel shows $\mathcal{M}_{[0,0]}, \mathcal{M}_{[0,1]}, \mathcal{M}_{[1,0]}, \mathcal{M}_{[1,1]}$ (left to right) after intersecting with $\mathbb{R}_{\geq 0}^{2\times 2}$ on the affine slice $\sum y_{ij}=1$.  
The tetrahedra represent the ambient slice, while the shaded regions represent the images of the parametrizations.}
\end{figure}
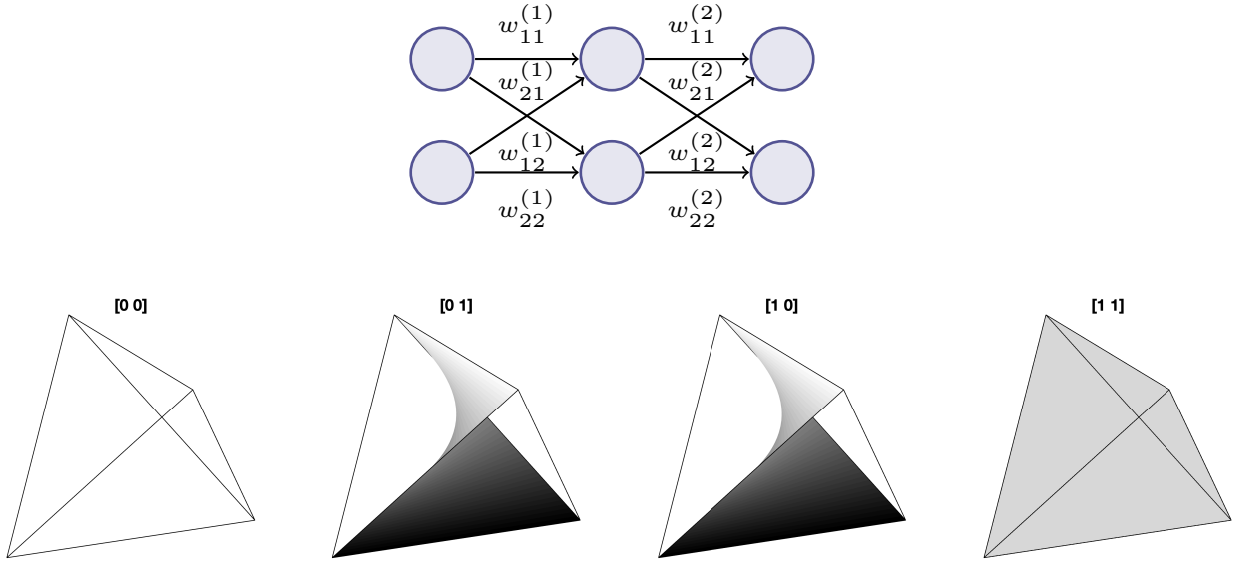 

\end{example}

\subsection{Multilinear structure}\label{sec:multilinear-structure}

For each activation pattern $A\in\{0,1\}^N$, the function computed by the network on $x\in R^A$ takes the form 
\begin{equation*}
f_\theta(x) 
= f^A_\theta(x) := g_L \circ  g_{L-1}^{A} \cdots \circ g_1^{A} (x) , 
\end{equation*}
where 
\begin{equation*}
g_\ell^{A}(y) 
= W^{(\ell)}_A y + b_A^{(\ell)}. 
\end{equation*}
This is a composition of affine maps with parameters defined as 
\begin{equation*}
W^{(\ell)}_A = \diag(A^\ell) W^{(\ell)}, \quad 
b_A^{(\ell)} = \diag(A^\ell) b^{(\ell)}, 
\end{equation*}
where we collect the activation pattern of the $\ell$th layer into $A^\ell = [A^\ell_1 ,\ldots, A^\ell_{n_\ell}]\in\{0,1\}^{n_\ell}$. 

For any given $X\in\mathbb{R}^{n_0\times m}$ and each activation pattern $\mathbf{A}\in \{0,1\}^{N\times m}$ we define 
\begin{equation*}
F^{\mathbf{A}}_X(\theta) : = [f_\theta^{\mathbf{A}_{:,1}}(x^{(1)}), \ldots, f_\theta^{\mathbf{A}_{:, m}}(x^{(m)})]\in\mathbb{R}^{n_L\times m}, \quad \theta \in \RR^{d_{\operatorname{par}}},
\end{equation*}
where $\mathbf{A}_{:, j}$ denotes the $j$th column of the matrix $\mathbf{A}$. Note that $F_X(\theta) = F^{\mathbf{A}}_X(\theta)$ for all $\theta\in S_X^{\mathbf{A}}$. Each component of the matrix $F_X(\theta)$ is a composition of functions which are linear in the parameters of each layer and thus the map
\begin{align}\label{eq:multilinear-map-version-1}
    \varphi_X^{\mathbf{A}}: \RR^{d_{\operatorname{par}}} \to  \RR^{n_L\times m};\; \theta \mapsto F_X^{\mathbf{A}}(\theta)
\end{align}
is multilinear in all parameters, as shown in the next proposition.

\begin{proposition}\label{prop:multilinear-structure}
For a given $\mathbf{A}\in\{0,1\}^{N\times m}$, the $L$-layer network map $\varphi_X^{\mathbf{A}}$ has $n_L\times m$ output coordinates that are polynomials in $\theta$ of degree $L$. Furthermore, the exponent of each parameter in any monomial is at most 1, i.e., the map is multilinear. 
\end{proposition}
\begin{proof}
It suffices to prove this for the case when $X$ contains only one input point $x$.

Note that if $L=1$, then the network has no hidden layers, and the map is simply $g_1$, which is linear in $\theta$. Let us consider the base case when we have one hidden layer; that is, $L=2$. Thus, $W^{(1)} \in \mathbb{R}^{n_1 \times n_0}$ and $b^{(1)} \in \mathbb{R}^{n_1}$ are the parameters. Let $A(x)\in\{0,1\}^{n_1}$ be the activation pattern. Then
    \[
        W^{(1)}_A = \diag(A)W^{(1)}, \, \text{ and }\, b^{(1)}_A = \diag(A)b^{(1)}, 
    \]
    and we have that
    \[
        \sigma(W^{(1)}x + b^{(1)}) = W_A^{(1)}x + b^{(1)}_A. 
    \]
    Composing with another linear map $g_2(y) = W^{(2)}y + b^{(2)}$, we get  a multilinear map of degree $L=2$ in~$\theta$, since the variables in $[W^{(1)}, b^{(1)}]$ and $[W^{(2)}, b^{(2)}]$ are disjoint. 

    Inducting on the number of layers, assume that the map $\varphi_X^{\mathbf{A}}$ is multilinear for a network with~$L$ layers. Adding a layer to obtain a network with $L+1$ layers introduces a linear map in new parameters, preserving multilinearity. Moreover, it increases the degree of each monomial by 1, so the degree of each monomial for an $L$-layer network is~$L$.
\end{proof}

\begin{example}\label{ex:simple-one-hidden-layer}
    Let $L=2$ and $n_0=n_1=n_2=2$. Moreover, assume that $b^{(1)} = b^{(2)} = 0$, i.e., that the network has no biases. Fix $X$ to contain a single data point $x=(x_1,x_2)\in\RR^2$, which lives in the linear region of the input space specified by the activation pattern $A=[1, 0]$. Then the image of $\theta\mapsto F_\theta(X)$ is parametrized as
    $$W^{(2)} \begin{bmatrix} 1 & 0 \\ 0 & 0 \end{bmatrix} W^{(1)} x = \begin{bmatrix}
        w_{11}^{(1)}w_{11}^{(2)} & w_{12}^{(1)}w_{11}^{(2)}\\
        w_{11}^{(1)}w_{21}^{(2)} & w_{12}^{(1)}w_{21}^{(2)}
    \end{bmatrix} \cdot \begin{pmatrix}
        x_1 \\ x_2
    \end{pmatrix}, \text{ where $W^{(\ell)}=(w^{(\ell)}_{ij})$.}
    $$
\end{example}

\subsection{Combinatorial structure}
When the dataset $X$ contains $m > 1$ points, 
we consider two cases:
\begin{itemize}
    \item All points $x\in X$ follow the same activation pattern $A$;
    \item Data points in $X$ follow different activation patterns.
\end{itemize}

\paragraph{The same activation pattern.}
First, suppose all points in $X$ follow the same activation pattern $A=[A^1,\ldots, A^{L-1}]$, where $A^\ell\in\{0,1\}^{n_\ell}$ denotes the activation pattern for the units in the $\ell$th layer. We define a \textit{path} in the network to be any tuple $p = (p_1, p_2, \ldots, p_{L-1})$ where $p_\ell \in [n_\ell]=\{1,\ldots, n_\ell\}$ for all $\ell\in [L-1]$. 
We call the path $p$ \textit{$A$-active} if $A^\ell_{p_\ell} = 1$ for all $\ell\in[L-1]$. Let $P_A$ denote the set of all $A$-active paths in our network.

Let $\mathbf{A}=[A,\ldots,A]\in\{0,1\}^{N\times m}$ be the activation pattern corresponding to $X$ in the parameter
region $S_X^{\mathbf{A}}$. 
Define the matrix $M(\theta)$ as the linear operator parametrizing $F^{\mathbf{A}}_X(\theta)$, so that
$$F^{\mathbf{A}}_X(\theta) = M(\theta) X + B(\theta),$$
where $B(\theta) = [b(\theta) \; b(\theta) \; \cdots \; b(\theta)]$
is a bias matrix with identical columns given by the vector~\(b(\theta)\). By the discussion in Section \ref{sec:multilinear-structure}, we know that $M(\theta) = W^{(L)} W_A^{(L-1)}\cdots W_A^{(1)}$ and $b(\theta) = \sum_{\ell =1}^L (W^{(L)} W_A^{(L-1)} \cdots W_A^{(\ell + 1)}) b^{(\ell)}_A$. We have the following description of the entries of $M(\theta)$ in terms of $A$-active paths.

\begin{proposition}\label{prop:comb-paths}
The $(i,j)$-entry of the matrix $M(\theta)$ is given as
\begin{align}\label{eq:path-parametrization}
M(\theta)_{ij} = \sum_{p = (p_1,\ldots, p_{L-1})\in P_A} w^{(1)}_{p_1j} w^{(2)}_{p_2p_1} w^{(3)}_{p_3p_2}\cdots w^{(L-1)}_{p_{L-1}p_{L-2}}w^{(L)}_{i p_{L-1}}.
\end{align}
\end{proposition}
\begin{proof}
Note that for any $\ell\in[L-1]$, the matrix $W^{(\ell)}_A$ is just the matrix $W^{(\ell)}$ with certain rows zeroed out. The non-zero rows of $W^{(\ell)}_A$ are precisely the rows labeled by $p_\ell$ such that $A^\ell_{p_\ell} = 1$. Hence, expanding matrix multiplication, we find that the $(i,j)$-entry of $W_A^{(L-1)}\cdots W_A^{(1)}$ is
$$ (W_A^{(L-1)}\cdots W_A^{(1)})_{ij} = \begin{cases}
    \sum w^{(1)}_{p_1 j}w^{(2)}_{p_2p_1}\cdots w^{(L-2)}_{p_{L-2}p_{L-3}}w^{(L-1)}_{ip_{L-2}} & \text{ if } i = p_{L-1}\text{ for some $p\in P_A$}\\
    0 & \text{ otherwise }, 
\end{cases}$$
where the sum is over all $(p_1,\ldots, p_{L-2})$ such that $p\in P_A$. Since the matrix $W^{(L)}$ has no zero rows, multiplying by it on the left gives the desired formula for $M_{ij}$.
\end{proof}
\begin{example}
Let $L = 3$, and $n_0 = n_1 = n_2 = n_3 = 3$. For $A = [(1, 1, 1), (1, 1, 0)]$, there are 6 active paths, so $P_A = \{(1,1), (1,2), (2,1), (2,2), (3,1), (3,2)\}$. Hence, 
$$M_{11} = 
w^{(1)}_{11}w^{(2)}_{11}w^{(3)}_{11} + w^{(1)}_{11}w^{(2)}_{21}w^{(3)}_{12} + 
w^{(1)}_{21}w^{(2)}_{12}w^{(3)}_{11} +
w^{(1)}_{21}w^{(2)}_{22}w^{(3)}_{12} +
w^{(1)}_{31}w^{(2)}_{13}w^{(3)}_{11} +
w^{(1)}_{31}w^{(2)}_{23}w^{(3)}_{12}.$$
\end{example}

\begin{example}\label{ex:M23-parametrization}
Let $L = 4$, and $n_0 = n_3 = n_4 = 3$ and $n_1 = n_2 = 2$. For the pattern $A = [(1, 1), (1, 0), (0, 0, 1)]$, there are only two active paths: $(1, 1, 3)$ and $(2, 1, 3)$. Therefore, 
$$M(\theta)_{23} = 
w^{(1)}_{13}w^{(2)}_{11}w^{(3)}_{31}w^{(4)}_{23} + w^{(1)}_{23}w^{(2)}_{12}w^{(3)}_{31}w^{(4)}_{23}.$$
\end{example}

For a fixed dataset $X$, the parametrization of $F^{\mathbf{A}}_X(\theta)$ in Example \ref{ex:M23-parametrization} is equivalent to the parametrization of the full linear network with $n_0=n_4=3, n_1=2$, and $n_2=n_3=1$. That is, all data points in $X$ follow the pattern $A=[(1,1),(1),(1)]$. The equivalence is up to renaming the parameters  $w^{(3)}_{3i}\leftrightarrow w^{(3)}_{1i}$ and $w^{(4)}_{j3}\leftrightarrow w^{(4)}_{j1}$ for all $i$ and $j$. This observation leads to the following lemma.

\begin{lemma}\label{lem:linear-net}
Let $\mathcal{N}$ be a ReLU network with widths $n_0,\ldots, n_L$ and let $X$ be a fixed dataset of $m$ affinely independent columns. Assume that all $x\in X$ follow the same activation pattern $A$, so $\mathbf{A}=[A,\ldots, A]$. 
The image of $F_X^{\mathbf{A}}(\theta)$ as a function of $\theta$ is equivalent to the image of $G_X(\theta)$ for the full linear network $\mathcal{L}$ with widths
$$
n_0,n^A_1,\ldots, n^A_{L-1},n_L, 
$$
where $n^A_\ell = |\{i : A^\ell_i = 1\}|$ is the number of $1$s in $A^\ell$.
\end{lemma}
\begin{proof}
Without loss of generality, we may assume that in the activation pattern $A=[A^1,\ldots, A^{L-1}]$, each $A^\ell$ is of the form $[1,\ldots, 1,0,\ldots, 0]$, i.e., all ones appear before all zeros. This assumption does not alter the parametrization of $F^{\mathbf{A}}_X$, only potentially relabels the parameters. Then we note that there is a natural bijection between $A$-active paths in $\mathcal{N}$ and all paths in $\mathcal{L}$, given by the identity mapping $p\mapsto p$. By Proposition \ref{prop:comb-paths} and the fact that $X$ is fixed, it follows that the parametrizations of $F_X^{\mathbf{A}}(\theta)$ for $\mathcal{N}$ and $F_X(\theta)$ for $\mathcal{L}$ are equivalent.
\end{proof}

\paragraph{Different activation patterns.}
Now suppose that data points in $X$ follow different activation patterns.
Then we may subdivide our dataset $X$ into $k$ blocks. Let $X = [X_1, \ldots, X_k]$ where for each $i=1, \ldots, k$ the block $X_i$ contains all data points of $X$ that follow the same activation pattern~$A_i$. Let $\mathbf{A}=\mathbf{A}(X)$ be the corresponding pattern. Let $\mathbf{A}_i$ denote the submatrix of $\mathbf{A}$ induced by the columns corresponding to the $i$th block. Let $M_i(\theta)$ denote the matrix in the parametrization~$\varphi_{X_i}^{\mathbf{A}_i}$ for the $i$th block only, i.e., $F_{X_i}^{\mathbf{A}_i}(\theta) = M_i(\theta)X_i + B_i(\theta)$. Then the parametrization of $F^{\mathbf{A}}_X(\theta)$~is: 
\begin{equation}\label{eq:param-multiple-blocks}
\theta \mapsto F_X^\mathbf{A}(\theta) = [M_1(\theta) X_1 + B_1(\theta) \mid M_2(\theta)X_2 +B_2(\theta) \mid 
\dots \mid M_k(\theta)X_k + B_k(\theta)]. 
\end{equation}
We are interested in studying the polynomial relationships on the entries of this image matrix.

\section{ReLU varieties}
\label{sec:relu-varieties}

In order to study polynomial constraints on the image of the map $\theta \mapsto F^{\mathbf{A}}_X(\theta)$, we will take the Zariski closure of the image, thereby obtaining a variety. We then work to determine the generating polynomials of the ideal defining this variety.

We introduce the necessary algebraic concepts and define ReLU output varieties, the main objects of our study. Fix an architecture of a ReLU network and let $X$ be a fixed but arbitrary dataset 
and consider some activation pattern $\mathbf{A}\in\{0,1\}^{N \times m}$.

We may take different perspectives to describe the set of functions that is expressible by a given ReLU network architecture: 1) the set of possible output arrays over a fixed but arbitrary set of input data points; 
2) the set of possible tuples of affine maps of the piecewise linear functions; 
3) the subset of parameters in a suitable ambient space that are attainable by the network.

\subsection{Algebraic notions}

We briefly recall some of the basic notions from algebraic geometry that underlie our approach. 
For accessible introductions, see \cite{cox2015ideals,grunwald2019mdl}.

For a given field $k$, the polynomial ring $k[x_1,\dots,x_n]$ is defined as the ring of polynomials with indeterminates $x_1,\ldots, x_n$ and coefficients in $k$. 
Given polynomials $f_1,\dots,f_m \in k[x_1,\dots,x_n]$, the set of their common zeros is denoted by 
$$
V(f_1,\dots,f_m) \;=\; \{\, a \in k^n \mid f_i(a) = 0 \;\; \forall i \in [m] \}
$$
and is called an \emph{algebraic variety}. 
Conversely, given a set $X \subseteq k^n$, its \emph{vanishing ideal} is the set of polynomials that vanish at every assignment $a\in X$ of the indeterminates $x_1,\ldots, x_n$, 
$$
I(X) \;=\; \{\, f \in k[x_1,\dots,x_n] \mid f(a)=0 \;\; \forall a \in X \}.
$$
Hilbert’s Nullstellensatz provides a precise correspondence between radical ideals\footnote{A radical ideal is an ideal $I$ such that if $f^n\in I$ for some power $n$, then also $f\in I$.} and algebraic varieties, linking algebraic and geometric descriptions.

In data science and machine learning contexts, algebraic varieties play a similar role to manifolds. They can be used, for instance, to model a data manifold or to model a set of feasible hypotheses. 
Characterizing an algebraic variety amounts to identifying \emph{generators} of its vanishing ideal, that is, a collection of polynomials that generates precisely the elements in the ideal via finite $k[x_1,\ldots,x_n]$-linear combinations. 
Moreover, the \emph{Minimum Description Length} principle connects to this 
picture: finding a compact set of generators corresponds to giving the ideal a minimal algebraic description. 

Finally, we recall that the \textit{Zariski topology} on \( k^n \) is the topology where the closed sets are \textit{varieties}. 
The Zariski closure of a set $X$ is defined to be the smallest variety that contains $X$. 

\subsection{Output variety} Recall that we are interested in the image of the parametric map
\[
\varphi_X^{\mathbf{A}}: \mathbb{R}^{d_{\operatorname{par}}} \to \mathbb{R}^{n_L \times m}; \quad \theta \mapsto F_X^{\mathbf{A}}(\theta).
\]
This image represents our 
\emph{model} over the input data \(X\), for any fixed activation pattern \(\mathbf{A}\). Taking the Zariski closure of this image in \(\mathbb{C}^{n_L \times m}\) yields the smallest complex algebraic variety containing it. This closure is important because it allows us to study the model using algebraic geometry tools: while the original image may be complicated and potentially non-closed in the usual topology (for an example, see \cite[Proof of Theorem 1]{Lim2022best}), the Zariski closure provides a well-defined, algebraically tractable object that captures the essential geometric and algebraic structure of the model’s output space.
\begin{definition}
Given a fixed input dataset $X$ and activation pattern $\mathbf{A}$, the \emph{ReLU output variety}, denoted $V_X^{\mathbf{A}}$, is defined as the Zariski closure of the image of $\varphi_X^{\mathbf{A}}$, that is, $V_X^{\mathbf{A}} := \overline{\operatorname{im} \varphi_X^{\mathbf{A}}}$.
\end{definition}

Let $I_X^{\mathbf{A}}$ denote the ideal defining this variety. 
The process of recovering the ideal of the Zariski closure of the image of $\varphi_X^{\mathbf{A}}$ from the parametric description is called \textit{implicitization} \cite[Chapter 4]{MichalekSturmfels2019InvitationNonlinearAlgebraTEXT}. This problem can be computationally challenging, as the standard approach relies on Gr\"obner basis computations \cite[Chapter 1]{MichalekSturmfels2019InvitationNonlinearAlgebraTEXT}, which (in the worst-case scenario) can grow doubly exponentially in the number of variables. The polynomials in the resulting ideal are called \textit{invariants}. Understanding these invariants is crucial, as it allows us to draw general conclusions about large models that cannot be obtained through symbolic computations. Moreover, knowing the defining equations of
a model significantly restricts the set of possible output values the model can produce for arbitrary datasets. 

\begin{example}\label{ex:ideal-I-simple}
Consider the model in Example \ref{ex:simple-one-hidden-layer}, whose parametrization is given as
$$\theta = (w_{ij}^{(\ell)}: i,j,\ell=1,2)\mapsto M(\theta)X \text{ where } M(\theta) = \begin{bmatrix}
        w_{11}^{(1)}w_{11}^{(2)} & w_{12}^{(1)}w_{11}^{(2)}\\
        w_{11}^{(1)}w_{21}^{(2)} & w_{12}^{(1)}w_{21}^{(2)}
    \end{bmatrix}.$$
\end{example}
If $X$ contains only one data point, then the variety $V_X^{\mathbf{A}}$ fills the entire space $\CC^2$, so there are no invariants. In other words, the outputs of this model are unrestricted. However, if $X$ contains two fixed but arbitrary data points $x^{(1)} = (x_{11}, x_{12}), x^{(2)}=(x_{21}, x_{22})\in\RR^2$ that are linearly independent, this is no longer true. Constructing the ideal
\begin{align*}
\langle 
&y_{11} - x_{11}w_{11}^{(1)}w_{11}^{(2)} - x_{12} w_{12}^{(1)}w_{11}^{(2)},\; 
y_{21} - x_{11}w_{11}^{(1)}w_{21}^{(2)} - x_{12}w_{12}^{(1)}w_{21}^{(2)} ,\\
&y_{12} - x_{21}w_{11}^{(1)}w_{11}^{(2)} - x_{22} w_{12}^{(1)}w_{11}^{(2)},\; 
y_{22} - x_{21}w_{11}^{(1)}w_{21}^{(2)} - x_{22}w_{12}^{(1)}w_{21}^{(2)} 
\rangle \in \CC[\theta, y_{11}, y_{12}, y_{21}, y_{22}]
\end{align*}
and eliminating the parameters $\theta$, we obtain the ideal $I_X^{\mathbf{A}}\subset \CC[y_{11}, y_{12}, y_{21}, y_{22}]$  generated by one polynomial: $y_{11}y_{22} - y_{12}y_{21}$. This ideal is \textit{determinantal}, since it is generated by the determinant of the matrix $Y=(y_{ij})$. The corresponding variety $V_X^{\mathbf{A}}$ is a curve in~$\CC^2$.

\subsection{Pattern variety}
Another natural algebraic object arises when we consider the parametrization 
$$\varphi^\mathbf{A}:\theta\mapsto [(M_1(\theta), B_1(\theta)) \mid (M_2(\theta), B_2(\theta)) \mid \cdots \mid (M_k(\theta), B_k(\theta))].$$ This parametrization is similar to~\eqref{eq:param-multiple-blocks}, except it does not depend on the dataset $X$. 
\begin{definition}
 Given a fixed activation pattern $\mathbf{A}$, the \emph{ReLU pattern variety}, denoted $U^{\mathbf{A}}$, is defined as the Zariski closure of the image of $\varphi^{\mathbf{A}}$, i.e.,
 $U^{\mathbf{A}} := \overline{\operatorname{im} \varphi^{\mathbf{A}}}$. 
\end{definition}
We denote the corresponding ideal by $J^{\mathbf{A}}$. Under the assumption that we have sufficiently many linearly independent data points in $X$, the ideals $I_X^{\mathbf{A}}$ and $J^{\mathbf{A}}$ are related. 

In Example \ref{ex:ideal-I-simple}, the ideal $J^{\mathbf{A}}$ is constructed by eliminating $\theta$ from the ideal
$$\langle 
y_{11} - w_{11}^{(1)}w_{11}^{(2)},\; 
y_{12} - w_{12}^{(1)}w_{11}^{(2)},\;
y_{21} - w_{11}^{(1)}w_{21}^{(2)},\;
y_{22} - w_{12}^{(1)}w_{21}^{(2)}
\rangle \in \CC[\theta, y_{11}, y_{12}, y_{21}, y_{22}].$$
Doing so, we obtain $J^{\mathbf{A}}=I_X^{\mathbf{A}} = \langle y_{11}y_{22} - y_{12}y_{21} \rangle$.

It is not true, in general, that the ideals $I_X^\mathbf{A}$ and $J^\mathbf{A}$ are equal. However, knowing polynomials in $J^\mathbf{A}$ can help us recover polynomials in $I_X^\mathbf{A}$ for any dataset $X$ with sufficiently many linearly independent vectors, when considering a ReLU network with no biases. In other words, knowing invariants that are independent of the data allows us to recover invariants for any fixed but arbitrary dataset. This idea is explored in Section \ref{sec:from-fun-to-pred}.

\subsection{CPWL parameter variety}

It is possible to consider other varieties, capturing different features of a model.
We highlight in particular the analysis of \cite{brandenburg2024the}, 
which considers as the features of interest the gradients and intercepts that appear in a piecewise linear decomposition of the functions represented by a ReLU network. These may be regarded as the parameters of the represented functions within an ambient space of continuous piecewise linear (CPWL) functions, analogous to the polynomial coefficients of the polynomials represented by networks with polynomial activation functions. 

To make this more concrete, consider a ReLU network with widths $n_0,n_1,\ldots, n_{L}$, 
which parametrizes a subset of continuous piecewise linear functions $\mathbb{R}^{n_0}\to\mathbb{R}^{n_L}$ as 
$$
f_\theta(x) = W^{(L)} \sigma(W^{(L-1)} \cdots \sigma(W^{(1)} x + b^{(1)}) \cdots + b^{(L-1)}) + b^{(L)}, 
$$
with parameter $\theta=((W^{(1)},b^{(1)}),\ldots, (W^{(L)},b^{(L)}))\in \mathbb{R}^{n_{1}\times (n_{0}+1)}\times\cdots\times\mathbb{R}^{n_L\times (n_{L-1}+1)}\cong \mathbb{R}^{d_{\operatorname{par}}}$. 
On the other hand, for any given $n$ and $m$ we may consider the continuous piecewise linear functions $\mathbb{R}^{n_0}\to\mathbb{R}^{n_L}$ that are parametrized as 
$$
g_\eta(x) = \max_{i\in[n]} \{ \langle s_i,x\rangle+a_i\}  - \max_{j\in[m]} \{\langle t_j, x\rangle + b_j \},
$$
with parameter $\eta= ((s_1,a_1),\ldots, (s_n,a_n),(t_1,b_1), \ldots, (t_m,b_m))\in \mathbb{R}^{n_L\times(n_0+1)}\times\cdots \times \mathbb{R}^{n_L\times (n_0+1)}\cong \mathbb{R}^q$, where the maximum is taken component-wise. 
The latter may be regarded as a canonical parametrization of continuous piecewise linear functions. Although this type of representation is not unique and the best possible choice of $n$ and $m$ is not obvious (see \cite{tran2023minimal}), for any given $k$ one can find $n$ and $m$ such that any continuous piecewise linear function with at most $k$ linear pieces can be expressed in this form. 

In particular, for any given $n_0,n_1,\ldots, n_{L}$ and suitable $n$ and $m$, there is a function 
$$
\varphi^{\operatorname{CPWL}} \colon \mathbb{R}^{d_{\operatorname{par}}}\to\mathbb{R}^q; \quad 
\theta\mapsto \eta\quad\text{ with }\quad
f_\theta(x) = g_{\varphi^{\operatorname{CPWL}}(\theta)}(x), \quad \text{for all $x$ and $\theta$}. 
$$

\begin{definition} 
We define the \emph{CPWL variety} as the Zariski closure of the image of $\varphi^{\operatorname{CPWL}}$. 
\end{definition} 

The subset of continuous piecewise linear functions that can be represented by a ReLU network can be described precisely in terms of the feasible set $\{\eta = \varphi^{\operatorname{CPWL}}(\theta)\colon \theta\in\mathbb{R}^{d_{\operatorname{par}}}\}\subseteq\mathbb{R}^q$. 
The result \cite[Theorem 3.2]{brandenburg2024the} shows that, for a choice of $n$ and $m$ depending on $n_0,n_1\ldots,n_{L-1}$ and $n_L=1$, this is a semi-algebraic subset of $\mathbb{R}^q$ defined by polynomial equations and inequalities of degree at most $L + 1$. 
The structure of the equations and inequalities
depending on the ReLU network architecture and the choice of $n$ and $m$ remains an open problem for further study that is related to the pattern varieties described above.

\section{Single-block ReLU output varieties}\label{sec:single-block}

In this section, we assume that all $m$ data points in the dataset $X$ follow the same activation pattern $A = [A^1, A^2, \ldots, A^{L-1}]$. Under this assumption, Lemma~\ref{lem:linear-net} implies that the corresponding pattern varieties reduce to determinantal varieties of $n_L \times n_0$ matrices of rank at most $r$, where $r = \min_{\ell\in [L-1]} \sum_{i\in[n_\ell]} A^\ell_i$.
These determinantal varieties are irreducible and their defining ideals are generated by the $(r+1)$-minors of the coordinate matrix $M$.

Accordingly, this section focuses on ReLU output varieties in the single-block case, which arise as the intersection of such determinantal varieties with linear spaces determined by the dataset. We treat the cases of networks with and without biases separately.

\subsection{Networks without biases} Assume that $b^{(\ell)}=0$ for all $\ell=1,\ldots, L$. We also assume that the points in $X$ are in general position. If $m \leq n_0$, this means that all data vectors are linearly independent, and if $m > n_0$, this means that any $n_0$ data vectors are linearly independent. Let $r>0$ denote the rank of $M(\theta) = W^{(L)}W_{A}^{(L-1)}\cdots W_{A}^{(1)}$ for generic $\theta$. Hence, all $(r+1)$-minors of $M(\theta)$ vanish. Note that $r=\min_{\ell\in [L-1]} \sum_{i\in[n_\ell]} A^\ell_i$ by Lemma \ref{lem:linear-net} and the fact that the matrices $W^{(\ell)}$ have disjoint parameters across layers. 

The assumption that the data points in $X$ are in general position ensures that the data vectors exhibit the maximal linear independence allowed by their dimensions. 
We point out that any dataset can be brought to general position by adding an arbitrarily small amount of noise. In particular, real-world data sets are typically in general position. 
Nonetheless, some datasets can violate the condition.
Algebraically, the failure of this assumption is not problematic: if the points are not in general position, then each additional linear dependency among the columns of $X$ yields a corresponding set of $n_L$ linear equations on the output variables. Thus, the ideal $I_X^\mathbf{A}$ in the next proposition simply acquires more linear generators, which may in turn reduce the complexity of its higher-degree part.

\begin{proposition}\label{prop:single-block-gens-linear}
The ideal $I_X^\mathbf{A}$ is generated by $n_L\cdot\max\{m-n_0, 0\}$ linear polynomials and $\binom{n_L}{r+1}\cdot\binom{\min\{n_0, m\}}{r+1}$ homogeneous polynomials of degree $r+1$.
\end{proposition}
\begin{proof}
Let $y_{11},\ldots, y_{n_L1},y_{12},\ldots, y_{n_L2},\ldots, y_{1m},\ldots, y_{n_Lm}$ denote the coordinates of the ambient space in which $V_X^\mathbf{A}$ lives. Then the parametrization of the variety is given as
$$ \varphi_X^\mathbf{A}:\theta \mapsto Y = \begin{bmatrix}y_{11} & y_{12} & \cdots & y_{1m}\\
    \vdots & \vdots & \ddots & \vdots \\
    y_{n_L1} & y_{n_L2} & \cdots & y_{n_Lm}
\end{bmatrix} = M(\theta)\cdot X,$$
where $M(\theta)$ is a matrix of rank $r$, and $X$ has rank $\min\{n_0, m\}$. 
Hence, the closure of the image of $\varphi_X^{\mathbf{A}}$ is a linear section of a determinantal variety and the defining ideal $I_X^{\mathbf{A}}$ is the sum of the ideal of $(r+1)$-minors of $M(\theta)$ and the ideal of linear dependencies imposed by $X$. Moreover, note that any rank-$r$ matrix in $\mathbb{C}^{\min\{n_0,m\} \times n_L}$ can be decomposed as $D_L \cdot D_{L-1} \cdots D_1$, where $D_\ell$ is an $n_{\ell} \times n_{\ell-1}$ matrix of rank $\min\left\{\sum_{i \in [n_{\ell-1}]} A_i^{\ell-1}, \sum_{i \in [n_{\ell}]} A_i^\ell \right\}$ for all $\ell \in [L-1]$. Since rank-$r$ matrices in $\mathbb{C}^{\min\{n_0,m\} \times n_L}$ form a Zariski open set inside the irreducible variety of matrices of rank at most~$r$, no other equations are needed to cut out the variety $V_X^\mathbf{A}$. We will now count these defining equations.

First, suppose that $m \leq n_0$. Since the vectors in $X$ are assumed to be in general position, they are linearly independent. Therefore, the only constraint on $Y$ is that it has rank at most $r$. Thus, the ideal $I_X^\mathbf{A}$ is generated by all $(r+1)$-minors of $Y$, so $I_X^\mathbf{A}$ has $\binom{n_L}{r+1}\binom{m}{r+1}$ generators of degree $r+1$.

Next, assume $m > n_0$. Then there are precisely $m - n_0$ linear dependencies among the columns of $X$. These dependencies can be expressed as $x^{(i)} = \sum_{j=m-n_0+1}^m \lambda_{ij} x^{(j)}$ for every $i\in[m-n_0]$, where the coefficients $\lambda_{ij}$ are unique. These dependencies on $X$ translate to the $(m-n_0)n_L$ equations in the coordinates of $Y$, namely:
$$y_{ki} = \sum_{j=m-n_0+1}^m \lambda_{ij} y_{kj},$$ where $i\in[m-n_0], k\in[n_L]$. Moreover, these linear equations induce algebraic dependencies between the  $(r+1)$-minors of $Y$, involving the first $m-n_0$ columns of $Y$. Hence, it suffices to only consider the minors of the submatrix $Y' = [y_{:,(m-n_0+1)}\cdots y_{:,m}]$ of $Y$. Hence, the ideal $I_X^\mathbf{A}$ is generated by $(m-n_0)n_L$ linear equations and $\binom{n_L}{r+1}\binom{n_0}{r+1}$ minors of $Y'$ of degree $r+1$.
\end{proof}

Recall that the dimension of all $\min\{n_0, m\} \times n_L$ matrices of rank at most $r$ is $\min\{n_0, m\}r + n_Lr - r^2$. This is precisely the dimension of the variety $V_X^\mathbf{A}$ in our case.

\begin{corollary}\label{cor:dim-one-block-linear}
    The variety $V_X^\mathbf{A}$ has dimension $\min\{n_0, m\}r + n_Lr - r^2$.
\end{corollary}
\begin{proof}
    Let $k = \min\{n_0,m\}$. The variety of all $n_L\times k$ matrices of rank at most $r$ is known to have dimension $d = kr + n_Lr - r^2$ in $\CC^{n_L\times k}$. Hence, it has codimension $n_Lk - d$ in the ambient space $\CC^{n_L\times m}$ of $V_X^\mathbf{A}$. The linear equations increase this codimension by $\max\{m-n_0, 0\}n_L$. Hence, $V_X^\mathbf{A}$ has dimension $mn_L - (n_Lk - d) - \max\{m-n_0, 0\}n_L = d$, as claimed.
\end{proof}

\begin{example}
Let $L = 2, n_0 = 3, n_1 = n_2 = 2$ and fix $A = [(1,0)]$, so $r = 1$. Let $m = 4$ and fix the dataset $X =
\left[\left(\begin{smallmatrix}
      3\\0\\-1\end{smallmatrix}\right),\:\left(\begin{smallmatrix}
      1\\1\\-1\end{smallmatrix}\right),\:\left(\begin{smallmatrix}
      3\\5\\5\end{smallmatrix}\right),\:\left(\begin{smallmatrix}
      0\\4\\0\end{smallmatrix}\right)\right]$.
The kernel of $X$ is spanned by the vector $(8, -18, -2, 7)^\top$, which defines the single linear dependency on the elements of $X$. This dependency is unique up to scaling. Therefore, the ideal $I_X^\mathbf{A}$ is generated by 2 linear and $\binom{3}{2}=3$ quadratic polynomials, namely:
\begin{align*}
J_X^\mathbf{A} = \langle &8\,y_{11}-18\,y_{12}-2\,y_{13}+7\,y_{14}, \;
8\,y_{21}-18\,y_{22}-2\,y_{23}+7\,y_{24}, \\
&\,y_{12}y_{23}-y_{13}y_{22},\;
\,y_{12}y_{24}-y_{14}y_{22},\;
\,y_{13}y_{24}-y_{14}y_{23}
\rangle.
\end{align*}
\end{example}

\subsection{Networks with biases}
In this subsection we consider ReLU networks with biases. 
We again assume that the points in $X$ are in general position. Let $r = \min_{\ell\in [L-1]} \sum_{i\in[n_\ell]} A^\ell_i$ as before. We have the following proposition. 

\begin{proposition}\label{prop:single-block-with-biases}
The ideal $I_X^\mathbf{A}$ is generated by $n_L\cdot\max\{m-n_0-1, 0\}$ linear polynomials and $\binom{n_L}{r+1}\binom{\min\{n_0, m-1\}}{r+1}$ homogeneous polynomials of degree $r+1$.
\end{proposition}
\begin{proof}

In this case, our parametrization is expressible as

$$ Y = \begin{bmatrix}y_{11} & y_{12} & \cdots & y_{1m}\\
    \vdots & \vdots & \ddots & \vdots \\
    y_{n_L1} & y_{n_L2} & \cdots & y_{n_Lm}
\end{bmatrix} = [b(\theta); M(\theta)]\cdot \begin{bmatrix}
    1\\ X 
\end{bmatrix} = M(\theta) X + B(\theta),$$
where $B = \begin{bmatrix} b(\theta)& b(\theta) & \cdots &b(\theta) \end{bmatrix}$ and $b(\theta) = \sum_{\ell =1}^L (W^{(L)} W^{(L)}_A\cdots W^{\ell +1}_A) b^\ell$. Since $M(\theta)$ generically has rank $r$, the columns of $M(\theta)$ span a linear space of dimension $r$. 

If $m \leq n_0 + 1$, then the columns of $X$ are affinely independent. Hence, the columns of $M(\theta)X + B(\theta)$ must span an affine linear space of dimension $r$. This is equivalent to requiring that the matrix $Y_1 = \begin{bmatrix} 1 & Y
\end{bmatrix}$ has rank $r+1$, i.e. all  $(r+2)$-minors of $Y_1$ must vanish. These minors are algebraically dependent; however, a subset—such as those obtained by fixing the first row and the last column—is sufficient to generate the ideal. The total number of such generators is $\binom{n_L}{r+1}\binom{m-1}{r+1}$.

If $m > n_0 + 1$, then the columns of $X$ are affinely dependent, and these dependencies can be uniquely written as $[1; x^{(i)}] = \sum_{j=m-n_0}^m \lambda_{ij} [1; x^{(j)}]$ for every $i\in[m-n_0-1]$, where the coefficients $\lambda_{ij}$ are unique. These dependencies on $X$ translate to the $(m-n_0-1)n_L$ equations in the coordinates of $Y$, namely:
$$y_{ki}= \sum_{j=m-n_0}^m \lambda_{ij} y_{kj},$$ where $i\in[m-n_0-1]$ and $k\in[n_L]$. Moreover, these linear equations induce algebraic dependencies between the $(r+2)$-minors of $Y_1$, which involve the first $m-n_0-1$ columns of $Y_1$. It then suffices to consider the minors of the submatrix $Y'_1 = \begin{bmatrix}
    1 &\cdots& 1\\
    y^{(m-n_0)}&\cdots& y^{(m)} 
\end{bmatrix}$ of $Y_1$, where $y^{(i)}$ denotes the $i$th column of $Y$. Hence, the ideal $I_X^\mathbf{A}$ is generated by $(m-n_0-1)n_L$ linear equations and $\binom{n_L}{r+1}\binom{n_0}{r+1}$ $(r+2)$-minors of $Y'_1$ of degree $r+1$.
\end{proof}

\begin{corollary}
    The variety $V_X^{\mathbf{A}}$ has dimension $\min\{n_0, m-1\}r + n_L(r+1) - r^2$.
\end{corollary}

\section{From functions to predictions}\label{sec:from-fun-to-pred}

In this section, for simplicity of exposition, we will only consider ReLU networks without biases, i.e., we assume that $b^{(\ell)}=0$ for all $\ell = 1,\ldots, L$. We consider a data set $X$ that is subdivided into $k$ blocks $X_1, \ldots, X_k$, according to $k$ activation patterns. We propose a method to transform a given polynomial in $J^\mathbf{A}$ into a polynomial in $I_X^{\mathbf{A}}$. We motivate this with the following example. 

\begin{example}\label{ex:big-ex-fun-to-pred} Let $L=3$ with $n_0=n_2=n_3=2$ and $n_1=1$.
Consider the dataset $X\in\mathbb{R}^{2\times 4}$ with four data points, divided into two blocks: $X_1 = [x^{(1)}, x^{(2)}] = \left[\left(\begin{smallmatrix} 1\\ 1 \end{smallmatrix}\right), \left(\begin{smallmatrix} 2\\ 3 \end{smallmatrix}\right) \right]$ and $X_2 = [x^{(3)}, x^{(4)}] = \left[\left(\begin{smallmatrix} 1\\ 2 \end{smallmatrix}\right), \left(\begin{smallmatrix} 3\\ 1 \end{smallmatrix}\right) \right]$.  Fix the pattern $A_1= [(1), (1, 0)]$  for data points in $X_1$ and $A_2= [(1), (0, 1)]$  for data points in~$X_2$. Then each block is parametrized by $W^{(i)}\mapsto M_i(\theta)X_i$, where:
$$
M_1(\theta)=\begin{bmatrix}
w_{11}^{(1)} w_{11}^{(2)} w_{11}^{(3)} & w_{12}^{(1)} w_{11}^{(2)} w_{11}^{(3)} \\
w_{11}^{(1)} w_{11}^{(2)} w_{21}^{(3)} & w_{12}^{(1)} w_{11}^{(2)} w_{21}^{(3)} 
\end{bmatrix}        \text{ and }    
M_2(\theta)=\begin{bmatrix}
w_{11}^{(1)} w_{21}^{(2)} w_{12}^{(3)} & w_{12}^{(1)} w_{21}^{(2)} w_{12}^{(3)} \\
w_{11}^{(1)} w_{21}^{(2)} w_{22}^{(3)} & w_{12}^{(1)} w_{21}^{(2)} w_{22}^{(3)} 
\end{bmatrix}. 
$$

Recall that the image of the map \( \varphi_X^\mathbf{A} \) is given by the matrix \( [M_1(\theta)X_1 \;|\; M_2(\theta)X_2] \). Let \( Y = [Y_1 \;|\; Y_2] \) be the matrix whose entries are indeterminates representing the coordinates of this image. Specifically, \( Y_1 = (y^{(1)}_{ij}) \) corresponds to the coordinates given by \( M_1(\theta)X_1 \), while \( Y_2 = (y^{(2)}_{ij}) \) corresponds to those given by \( M_2(\theta)X_2 \). The polynomial equality constraints on the image \( F_X^\mathbf{A}(\theta) \), as \( \theta \) varies over $S_X^\mathbf{A}$ with \( \mathbf{A} = [A_1, A_1, A_2, A_2] \), are given by the ideal \( I_X^\mathbf{A} \) of the output variety~\( V_X^\mathbf{A} \). These constraints fall into three types:

\begin{enumerate}
    \item[(1)] the constraints in the variables of $Y_1$ (single quadric $q_1$ from Proposition \ref{prop:single-block-gens-linear}); 
    \item[(2)]  the constraints in the variables of $Y_2$ (single quadric $q_2$ from Proposition \ref{prop:single-block-gens-linear}); 
    \item[(3)]  the constraints that involve the variables in $Y_1$ and $Y_2$ simultaneously.
\end{enumerate}

The quadrics from (1) and (2) are the minors of $Y_1$ and $Y_2$, respectively:
\[
q_1 = y^{(1)}_{11}y^{(1)}_{22} - y^{(1)}_{12}y^{(1)}_{21}, \quad
q_2 = y^{(2)}_{11}y^{(2)}_{22} - y^{(2)}_{12}y^{(2)}_{21}.
\]
They belong to the generating sets of both \( I_X^\mathbf{A} \) and \( J^\mathbf{A} \).  To find the constraints of $I_X^{\mathbf{A}}$ of type (3), we first compute these constraints for the ideal $J^\mathbf{A}$ of the pattern variety $U^\mathbf{A}$. Let $M_1 = (m_{ij}^{(1)})$ and $M_2 = (m_{ij}^{(2)})$ be two matrices of indeterminates labeling the coordinates of the image of $\varphi^
\mathbf{A}$. The constraints involving the entries of both $M_1 = (m_{ij}^{(1)})$ and $M_2 = (m_{ij}^{(2)})$ in $J^\mathbf{A}$ are:
\begin{align*}
    &f_1=m^{(1)}_{11}m^{(2)}_{12}-m^{(1)}_{12}m^{(2)}_{11},\; f_2 = m^{(1)}_{11}m^{(2)}_{22}-m^{(1)}_{12}m^{(2)}_{21}, \\
    &f_3=m^{(1)}_{21}m^{(2)}_{22}-m^{(1)}_{22}m^{(2)}_{21},\; f_4=m^{(1)}_{21}m^{(2)}_{12}-m^{(1)}_{22}m^{(2)}_{11}. 
\end{align*}

Each of these polynomials gives rise to a unique polynomial in the ideal $I_X^\mathbf{A}$. 

For example, consider the polynomial $f_3=m^{(1)}_{21}m^{(2)}_{22}-m^{(1)}_{22}m^{(2)}_{21}$. Note that by definition $Y_i = M_i X_i$, so $M_i = Y_i X_i^{-1}$ for $i=1,2$, since both $X_1$ and $X_2$ are invertible. Making the substitutions
$$\left[\begin{smallmatrix}
    m^{(1)}_{11} & m^{(1)}_{12}\\
    m^{(1)}_{21} & m^{(1)}_{22}
\end{smallmatrix}\right] \mapsto \left[\begin{smallmatrix}
    3y^{(1)}_{11} - y^{(1)}_{12}& -2y^{(1)}_{11}+ y^{(1)}_{12}\\
    3y^{(1)}_{21} - y^{(1)}_{22}& -2y^{(1)}_{21}+ y^{(1)}_{22}
\end{smallmatrix}\right] \text{ and } \left[\begin{smallmatrix}
    m^{(2)}_{11} & m^{(2)}_{12}\\
    m^{(2)}_{21} & m^{(2)}_{22}
\end{smallmatrix}\right] \mapsto \frac{1}{5}\left[\begin{smallmatrix}
    -y^{(2)}_{11} + 2y^{(2)}_{12}& 3y^{(2)}_{11} - y^{(2)}_{12}\\
    -y^{(2)}_{21} + 2y^{(2)}_{22}& 3y^{(2)}_{21} - y^{(2)}_{22}
\end{smallmatrix}\right] $$
in the polynomial $f_1$, we obtain a polynomial  
$$g_3 = -\frac{1}{5}(7y^{(1)}_{21}y^{(2)}_{21} - 2y^{(1)}_{22}y^{(2)}_{21} + y^{(1)}_{21}y^{(2)}_{22} - y^{(1)}_{22}y^{(2)}_{22})\in I_X^{\mathbf{A}}.$$
In a similar fasion, we obtain the polynomials $g_1,g_2$, and $g_4$ from $f_1, f_2$ and $f_4$, respectively. In the end, we obtain the ideal of the output variety $I_X^\mathbf{A} = \langle q_1, q_2, g_1, g_2, g_3, g_4\rangle$.

\end{example}

This example motivates a general procedure to translate the polynomials in $J^\mathbf{A}$ to the polynomials $I_X^\mathbf{A}$. Let $X=[X_1, \ldots, X_k]\in \RR^{n_0\times m}$ be a dataset, subdivided into $k$ blocks. 
For each $i\in[k]$, we assume that:
\begin{itemize}
    \item  $|X_i| = n_0$,
    \item all points in $X_i$ follow the same activation pattern,
    \item all points in $X_i$ are linearly independent.
\end{itemize}
Note that if $|X_i|>n_0$, we may restrict our attention to a subset of $n_0$ linearly independent vectors, satisfying the above assumptions; see Remark~\ref{rem:more-than-n0} below. Let $\mathbf{y} = \{y^{(i)}_{j_1j_2}: i\in[k], j_1\in[n_L], j_2\in[n_0]\}$ and  $\mathbf{m} = \{m^{(i)}_{j_1j_2}: i\in[k], j_1\in[n_L], j_2\in[n_0]\}$. 
Consider the ring homomorphism
\[
\psi: \mathbb{R}[\mathbf{y}] \to \mathbb{R}[\mathbf{m}], \quad
\operatorname{diag}(Y_1, Y_2, \dots, Y_k) \mapsto 
\operatorname{diag}(M_1, M_2, \dots, M_k) \cdot \operatorname{diag}(X_1, X_2, \dots, X_k),
\]
which can be treated as a linear change of coordinates, defined by the change-of-coordinates matrix $\diag(X_1, \ldots, X_k)^\top$. 
Here we write $\operatorname{diag}(Y_1, Y_2, \dots, Y_k)$ for the block diagonal matrix with blocks $Y_1,\ldots, Y_k$ along the diagonal. 
Since each $X_i$ is invertible, the matrix $\diag(X_1, \ldots, X_k)^\top$ is also invertible. Therefore, $\psi$ is a bijection, with the inverse
\[
\psi^{-1}: \mathbb{R}[\mathbf{m}] \to \mathbb{R}[\mathbf{y}], \quad
\operatorname{diag}(M_1, M_2, \dots, M_k) \mapsto 
\operatorname{diag}(Y_1, Y_2, \dots, Y_k) \cdot 
\operatorname{diag}(X_1^{-1}, X_2^{-1}, \dots, X_k^{-1}). 
\]

We have the following proposition. 
\begin{proposition}\label{prop:from-fun-to-pred}
    Any polynomial $f\in J^{\mathbf{A}}$ gives rise to a unique polynomial $g = \psi^{-1}f\in I_X^{\mathbf{A}}$. 
\end{proposition}
\begin{proof}
Let $f\in\RR[\mathbf{m}]$ be a polynomial in $J^\mathbf{A}$, which means that $f(M_1, \ldots, M_k) = 0$. In other words, $f$ vanishes on the parametrization of $U^\mathbf{A}$. Moreover, $Y_i = M_i X_i$ for $i\in[k]$. We~have $$g(Y_1, \ldots, Y_k) = \psi^{-1}f(M_1,\ldots, M_k)= f(Y_1X_1^{-1}, \ldots, Y_kX_k^{-1}) = f(M_1,\ldots, M_k) = 0,$$
which means that $g$ vanishes on the parametrization of $V_X^\mathbf{A}$, and thus $g\in I_X^\mathbf{A}$, as claimed.
\end{proof}

\begin{remark}
When the data points in some $X_i$ are not linearly independent, Proposition \ref{prop:from-fun-to-pred} no longer holds. Fix the ReLU network with the same architecture as in Example \ref{ex:big-ex-fun-to-pred}. Let $X_1 = \left[\left(\begin{smallmatrix} 1\\ 1 \end{smallmatrix}\right), \left(\begin{smallmatrix} 2\\ 2 \end{smallmatrix}\right) \right]$ and $X_2 = \left[\left(\begin{smallmatrix} 1\\ 0 \end{smallmatrix}\right), \left(\begin{smallmatrix} 2\\ 1 \end{smallmatrix}\right) \right]$.  Fix the same patterns $A_1= [(1), (1, 0)]$ and $A_2= [(1), (0, 1)]$. The ideal of the pattern variety $U^\mathbf{A}$ is still generated by six polynomials
    $J^{\mathbf{A}} = \langle q_1, q_2, f_1, f_2, f_3, f_4\rangle.$
    However, since the data points in $X_1$ are not linearly independent, the ideal $I_X^{\mathbf{A}}$ is generated by only three polynomials
    $$I_X^{\mathbf{A}} = \langle 2 y^{(1)}_{11} - y^{(1)}_{12},\;\;\; 2 y^{(1)}_{21} - y^{(1)}_{22},\;\;\; y^{(2)}_{11}y^{(2)}_{22} - y^{(2)}_{12}y^{(2)}_{21}\rangle,$$
    two of which reflect the linear dependency between $\left(\begin{smallmatrix} 1\\ 1 \end{smallmatrix}\right)$ and $\left(\begin{smallmatrix} 2 \\ 2 \end{smallmatrix}\right)$.
\end{remark}

\begin{remark}
\label{rem:more-than-n0}
    When a data block contains more than $n_0$ vectors, we can still obtain the ideal of the ReLU variety using Proposition \ref{prop:single-block-gens-linear}.
    Without loss of generality, suppose $X_1$ consists of $m_1 > n_0$ vectors. Choose $n_0$ linearly independent vectors from this block, denote this subset by $X'_1$, and obtain the ideal $I_{[X'_1\;|\; X_2]}^\mathbf{A}$ from the ideal $J^\mathbf{A}$. Adding $n_L(m_1 - n_0)$ linear polynomials that capture the dependencies among the vectors in $X_1$, we obtain the desired ideal $I_X^{\mathbf{A}}$. 
    
    For example, fix the same ReLU network with architecture as in Example \ref{ex:big-ex-fun-to-pred}. Let $X_1 = \left[\left(\begin{smallmatrix} 3\\ 1 \end{smallmatrix}\right), \left(\begin{smallmatrix} 1\\ 1 \end{smallmatrix}\right), \left(\begin{smallmatrix} 1\\ 3 \end{smallmatrix}\right)  \right]$ and $X_2 = \left[\left(\begin{smallmatrix} 1\\ 0 \end{smallmatrix}\right), \left(\begin{smallmatrix} 2\\ 1 \end{smallmatrix}\right) \right]$.  Fix $A_1= [(1), (1, 0)]$ and $A_2= [(1), (0, 1)]$. Then
    $$
    I_{\left[\left(\begin{smallmatrix} 3\\ 1 \end{smallmatrix}\right), \left(\begin{smallmatrix} 1\\ 1 \end{smallmatrix}\right) \; | \; \left(\begin{smallmatrix} 1\\ 0 \end{smallmatrix}\right), \left(\begin{smallmatrix} 2\\ 1 \end{smallmatrix}\right) \right]}^\mathbf{A} + \langle y^{(1)}_{11} - 4 y^{(1)}_{12} + y^{(1)}_{13}, \; y^{(1)}_{21} - 4 y^{(1)}_{22} + y^{(1)}_{23} \rangle 
    $$
    is the ideal of the ReLU variety $V_X^{\mathbf{A}}$.
    
\end{remark}

\section{Two-block ReLU pattern varieties}\label{sec:two-blocks}

In this section, we study the generators of the ideal $J^\mathbf{A}$ of the pattern variety in the case when~$\mathbf{A}$ contains two distinct input activation patterns: $A_1$ and $A_2$. That is, we are interested in the polynomial constraints that hold for the image of the parametrization $\theta\mapsto [M_1(\theta)\;|\;M_2(\theta)]$. Let $r_i=\rank(M_i(\theta))=\min_{\ell\in[L-1]}\sum_{j\in[n_\ell]}A_{ij}^\ell$ for $i=1, 2$. Throughout the section, we assume that the network has no biases, i.e., $b^{(\ell)} = 0$ for all $\ell\in [L]$. As in the previous section, we denote by $M_i = (m^{(i)}_{j_1j_2})$ the matrix of indeterminates in the image space, corresponding to the entries of the parametrized matrix $M_i(\theta)$.

Let $P_1$ denote the set of all $A_1$-active paths and let $P_2$ denote the set of all $A_2$-active paths. We say that two paths $p = (p_1, \ldots, p_{L-1})$ and $q = (q_1, \ldots, q_{L-1})$ \textit{intersect} if $p_i=q_i$ for some $i=1,\ldots,L-1$. Otherwise, we say that the intersection is empty, and denote it by $p\cap q = \varnothing$.

\begin{proposition}\label{prop:disjoint-paths-ideal}
If $p_1\cap p_2=\varnothing$ for any $p_1\in P_1$ and $p_2\in P_2$, then 
$$J^{\mathbf{A}} = \langle(r_1+1)\text{-minors of $M_1$}\rangle + \langle(r_2+1)\text{-minors of $M_2$}\rangle.$$
\end{proposition}
\begin{proof}
Since $M_1(\theta) = W^{(L)}W^{(L-1)}_{A_1}\cdots W^{(1)}_{A_1}$ is the product of matrices with disjoint parameter entries, the only polynomial constraints $J ^{\mathbf{A}}$ involving $M_1$ enforce the rank condition on $M_1$. 
That is, the only polynomials in \( J^{\mathbf{A}} \) involving the variables associated to  $M_1=(m^{(1)}_{ij})$ are the $(r_1+1)$-minors of $M_1$. 
Similarly, the only polynomials in the variables $M_2=(m^{(2)}_{ij})$ are the $(r_1+2)$-minors of $M_2$. 

Since $A_1$-active and $A_2$-active paths do not intersect, the parametrizations $\theta\mapsto M_1(\theta)$ and $\theta\mapsto M_2(\theta)$ involve disjoint sets of parameters $(w^{(\ell)}_{ij})$, as seen from Proposition~\ref{prop:comb-paths}. Therefore, there are no polynomials in $J^\mathbf{A}$ that involve variables from both blocks. 
\end{proof}

\subsection{Shallow networks}
\label{sec:shallow}
In this section, we focus on ReLU networks with only one hidden layer, referred to as \textit{shallow} ReLU networks. 
In this case, $L=2$ and $n_0, n_1, n_2$ are the input dimension, the width of the hidden layer, and the output dimension, respectively. 
For generic $\theta$, let \(\rank(M_1(\theta)) = r_1 = \sum_jA_{1j}\) and \(\rank(M_2(\theta)) = r_2 = \sum_jA_{2j}\), so the first and second blocks have \(r_1\) and \(r_2\) active units in the hidden layer, respectively. Let \(R_1\) and \(R_2\) be their sets of active neurons, with \(S = R_1 \cap R_2\) containing the $s = \sum_jA_{1j} A_{2j}$ neurons that are active in both blocks. 
Define $t := r_1 + r_2 - 2s$. Without loss of generality, assume that $r_1 + r_2 - s = n_1$, i.e., every neuron in the hidden layer is active in some block. 
For the network in Figure~\ref{fig:shallow-networks}, $r_1=r_2=3$, $s=2$, and $t=2$.

\begin{figure}[H]
\centering
\begin{tikzpicture}[scale=0.8]
    \node at (0, 3) {$n_0$};
    \node at (2, 3) {$n_1$};
    \node at (4, 3) {$n_2$};
    
    \foreach \y in {2, 1, 0, -1} {
        \fill (0, \y) circle (3pt);
        \fill (4, \y) circle (3pt);
    }

    \foreach \y in {2, 1, 0, -1} {
        \fill (2, \y) circle (3pt);
    }

    \draw[orange, thick] (1.8, 1.5) rectangle (2.2, -0.5);
    \draw[magenta, thick] (1.6, 1.55) rectangle (2.4, -1.5);
    \draw[cyan, thick] (1.7, 2.5) rectangle (2.3, -0.55);
    
    \node[orange] at (1.3, 0.5) {\small $S$};
    \node[cyan] at (2.7, 2) {\small $R_1$};
    \node[magenta] at (2.7, -1) {\small $R_2$};

    \node at (0, -2) {\text{input}};
    \node at (2, -2) {\text{hidden}};
    \node at (4, -2) {\text{output}};

\end{tikzpicture}    
\caption{Shallow ReLU network with $n_0=n_1=n_2=4$ and two blocks given by the patterns $A_1 = [1,1,1,0]$ and $A_2 = [0,1,1,1]$.}
    \label{fig:shallow-networks}
\end{figure}
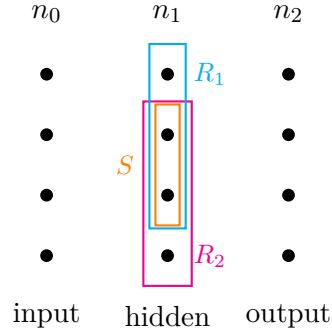

\begin{theorem}\label{thm:shallow-invariants} The ideal $J^\mathbf{A}$ contains: 
     \begin{enumerate}
        \item $(r_1+1)$-minors of $M_1$;
        \item $(r_2+1)$-minors of $M_2$;
        \item $(n_1+1)$-minors of $[M_1\;|\;M_2]$ and $[M_1^\top \;|\; M_2^\top]$; 
        \item $(t+1)$-minors of $M_1 - M_2$. 
     \end{enumerate}
\end{theorem} 

\begin{proof} 
We will show that all four types of polynomials in the statement belong to the ideal \( J^{\mathbf{A}} \). The inclusion of type 1 polynomials in \( J^{\mathbf{A}} \) follows from the fact that \( M_1 \) is a rank-\( r_1 \) matrix that can be written as a product of matrices with disjoint parameter entries. 
By a similar argument, the equations of type 2 also belong to the ideal. 
These first two types of equations correspond to polynomial constraints that hold independently within each of the two blocks. 

Next, note that the matrix $[M_1(\theta) \;|\; M_2(\theta)]$ parametrizing our model can be written as $[A(\theta) + C(\theta) \;|\; B(\theta) + C(\theta)]$, where $A(\theta)$ is the matrix parametrizing the model induced by $R_1\setminus S$, $B(\theta)$ is the matrix parametrizing the model induced by $R_2\setminus S$, and $C(\theta)$ is the matrix parametrizing the model induced by~$S$. All these models are full linear models, so $\rank A(\theta) = r_1 - s,$ $\rank B(\theta) = r_2 - s,$ and $\rank C(\theta) = s$ for generic $\theta$. Moreover, by \eqref{prop:comb-paths}, the parameters in $A(\theta), B(\theta),$ and $C(\theta)$ are pairwise disjoint.
We may write $[M_1(\theta) \;|\; M_2(\theta)] = [A(\theta) \;|\; 0] + [0 \;|\; B(\theta)] + [C(\theta) \;|\; C(\theta)]$, so $\rank [M_1(\theta)\;|\; M_2(\theta)]\leq \rank A(\theta) + \rank B(\theta) + \rank C(\theta) = r_1 + r_2 - s = n_1$ for any $\theta$. The argument is the same for $[M_1^\top(\theta)\;|\; M_2^\top(\theta)]$. We conclude that the equations of type 3 are in the ideal $J^\mathbf{A}$.

Finally, observe that $M_1(\theta) - M_2(\theta) = A(\theta) - B(\theta)$ has rank at most $t = r_1+r_2-2s$ for any $\theta$. This means that all $(t+1)$-minors of $M_1 - M_2$ have to vanish on the parametrization, and hence the equations of type 4 are in the ideal $J^\mathbf{A}$. 
\end{proof}

\begin{remark}
We note that Theorem~\ref{thm:shallow-invariants} also holds for the case when the two blocks have disjoint paths, like in Proposition~\ref{prop:disjoint-paths-ideal}. 
Indeed, in this case, we have that $S=\varnothing$ and $r_1 + r_2 = t = n_1$. Consequently, any $(t+1)$-minor or $(n_1+1)$-minor of $[M_1\;|\; M_2]$ must involve either at least $r_1 + 1$ columns from $M_1$ or at least $r_2 + 1$ columns from $M_2$. Moreover, every polynomial of type 4 is a linear combination of certain minors of  $[M_1\;|\; M_2]$. Hence, all polynomials of types 3 and 4 are already contained in the ideal generated by the polynomials of types 1 and 2, and hence redundant.
\end{remark}

\begin{example}
Consider the network in Figure \ref{fig:shallow-networks}. The ideal $J^\mathbf{A}$ is generated by two quartic determinants of $M_1=(m^{(1)}_{ij})$ and $M_2=(m^{(1)}_{ij})$, and sixteen cubic minors of $M_1 - M_2$. It can be confirmed computationally that there are no other generators, and the corresponding ideal has dimension~26. 
This means that the parametrization map defines a 26-dimensional variety inside of $\mathbb{C}^{32}$. Moreover, since the Jacobian achieves full rank at a real point in the parameter space, the same dimension holds when $U^{\mathbf{A}}$ is regarded as a real variety. Intuitively, although the variety lies in a 32-dimensional ambient space, the real outputs of the network locally sweep out a 26-dimensional subset of~$\mathbb{R}^{32}$.
\end{example}

\begin{example}
Consider the shallow network with $n_0=n_2=4$ and $n_1=3$. Suppose the two blocks are given by the activation patterns $[1,1,0]$ and $[0,1,1]$, so $t = 2$. The ideal $J^\mathbf{A}$ has dimension 21 and is minimally generated by 48 cubics and 40 quartics. Among the 48 cubics, $16\times 2 = 32$ are the $3$-minors corresponding to single blocks (polynomials of types 1 and 2), while the remaining 16 are the 3-minors of type 4. 

The 40 quartics are the algebraically independent $(n+1)$-minors of $[M_1 \;|\; M_2]$ and $[M_1^\top \;|\; M_2^\top]\;$, corresponding to type 3 polynomials. The algebraic dependencies arise because all 3-minors of $M_1$ and $M_2$ are in the ideal, so it suffices to consider the 36 4-minors that involve two columns from each block. However, 16 relations must be subtracted to account for dependencies introduced by the cubic generators of type 4. Hence, each of the matrices $[M_1 \;|\; M_2]$ and $[M_1^\top \;|\; M_2^\top]$ contributes 20 algebraically independent 4-minors.
\end{example}

\begin{conj}\label{conj:sufficiency}
The ideal $J^{\mathbf{A}}$ is generated by the polynomials in Theorem \ref{thm:shallow-invariants}.
\end{conj}

\begin{remark}\label{rmk:rowspace-intersection}
  To show that the four types of polynomials suffice to generate the entire ideal~$J^\mathbf{A}$, it is enough to prove that for any pair $(M_1, M_2) \in \CC^{n_2 \times n_0} \times \CC^{n_2 \times n_0}$ satisfying the four types of relations in Theorem~\ref{thm:shallow-invariants}, there exist parameter matrices $A$, $B$, and $C$ of size $n_L \times n_1$ and ranks at most $r_1 - s$, $r_2 - s$, and $s$, respectively, such that $M_1 = A + C$ and $M_2 = B + C$. When $\dim(\operatorname{rowspace} M_1 \cap \operatorname{rowspace} M_2) = s$, 
constructing such matrices is straightforward provided we make an additional assumption. 
Let $\{v_1,\ldots,v_s\}$ be an orthonormal basis for the common row space 
$V := \operatorname{rowspace} M_1 \cap \operatorname{rowspace} M_2$, 
and define $N := [v_1 \cdots v_s]^\top$. 
Let $P := N^*N$ denote the orthogonal projection onto~$V$, 
where $N^*$ is the conjugate transpose (or the transpose if working over the reals). 
We assume that $M_1$ and $M_2$ agree on $V$, i.e. $(M_1 - M_2)v = 0$ for all  $v \in V$, or equivalently, $(M_1 - M_2)P = 0.$ Under this assumption we obtain $M_1P = M_2P$, and we may define $C := M_1P = M_2P,$
which lies entirely in the common row space and has rank $s$. Then define $A := M_1 - C = M_1(I - P)$ and $B := M_2 - C = M_2(I - P)$. By construction, row spaces of $A$ and $B$ lie in the orthogonal complement of $V$, and so they are disjoint from that of $C$. Since $\rank M_1 \leq r_1$ and $\rank M_2 \leq r_2$, it follows that $\rank A \leq r_1 - s$ and $\rank B \leq r_2 - s$. A similar construction applies when $\dim(\operatorname{colspace} M_1 \cap \operatorname{colspace} M_2) = s$, by projecting onto a shared column space instead. When neither of these conditions is satisfied, constructing the parameter matrices~is~trickier.
\end{remark}
\begin{example}
    Let $s = 2, r_1 = 3, r_2 = 4$, and consider 
    $$
    M_1 = \begin{bmatrix}
        1 & 0 & 0 & 0 \\
        0 & 1 & 0 & 0\\
        0 & 0 & 1 & 0\\
        0 & 0 & 0 & 0
    \end{bmatrix} \text{ and } M_2 = \begin{bmatrix}
        1 & 0 & 0 & 0 \\
        0 & 0 & 0 & 0\\
        0 & 0 & 0 & 0\\
        0 & 0 & 0 & 1
    \end{bmatrix} . 
    $$
The row spaces (and column spaces) of $M_1$ and $M_2$ overlap in a 1-dimensional space, so the construction in Remark \ref{rmk:rowspace-intersection} would force $A$ to have rank greater than $1 = r_1 - s$. However, it is still possible to construct the matrices $A, B,$ and $C$ with appropriate ranks. Letting $A = \diag(0,1,0,0)$ and $B = \diag(0,0,-1,1)$, we find that $C = \diag(1,0,1,0)$. This means that $M_1 = A + C$ and $M_2 = B + C$, while $\rank A \leq r_1 - s, \rank B\leq r_2 -s, \rank C \leq s$, as desired.
\end{example}

\subsection{Deep networks}\label{sec:deep-networks}
The results in the previous section generalize to the ReLU networks with multiple hidden layers, referred to as \textit{deep} ReLU networks. 
As before, let $n_0, n_1, \ldots, n_L$ denote the input dimension, the widths of the hidden layers, and the output dimension, respectively. Let $r_1$ and $r_2$ be the ranks of the fully connected networks induced by the active neurons in the first and second blocks, respectively. Let $s$ be the rank of the network induced by the active neurons in both blocks. Without loss of generality, assume that $r_1 + r_2 - s = n_1$. We define a \textit{path network} associated with a set of paths $P$ as a ReLU network parametrized by $\theta \mapsto M(\theta)X$, where $M(\theta)$ follows \eqref{eq:path-parametrization} with summation restricted to $P$:
\[
M(\theta)_{ij} = \sum_{p = (p_1,\ldots, p_{L-1}) \in P} w^{(1)}_{p_1j} w^{(2)}_{p_2p_1} w^{(3)}_{p_3p_2} \cdots w^{(L-1)}_{p_{L-1}p_{L-2}} w^{(L)}_{i p_{L-1}}.
\]
The \textit{rank} of a path network is defined as the rank of its associated parameter matrix $M(\theta)$ for a generic parameter $\theta$. Importantly, a path network is not necessarily equivalent to a fully connected network, as its structure is constrained by the specified set of paths.

Let $R_1$ be the set of $A_1$-active paths and $R_2$ the set of $ A_2 $-active paths. Let $ S $ denote the set of paths active in both blocks.  We define $r_a$ as the rank of the path network determined by $R_1 \setminus S$, and $r_b$ as the rank of the path network determined by $R_2 \setminus S$. We set $t := r_a + r_b$.

\begin{example}\label{ex:deep-not-fully-connected}
Consider the network in Figure \ref{fig:deep-networks}. Here, $L=4$ and $n_\ell=2$ for all $\ell=0, \ldots, 4$. The activation patterns corresponding to the two blocks are $A_1 = [(1,1),(0,1),(1,1)]$ and $A_2 = [(0,1),(1,1),(0,1)]$, respectively. The set $R_1\setminus S$ consists of three paths: $(1,2,1), (2,2,1),$ and $(1,2,2)$, so the corresponding path network is parametrized~as

\vspace{1em}
\noindent\resizebox{\textwidth}{!}{$
\begin{bmatrix}
    w^{(1)}_{11} w^{(2)}_{21} w^{(3)}_{12} w^{(4)}_{11}
    + w^{(1)}_{21} w^{(2)}_{22} w^{(3)}_{12} w^{(4)}_{11}
    + w^{(1)}_{11} w^{(2)}_{21} w^{(3)}_{22} w^{(4)}_{12}
    &
    w^{(1)}_{12} w^{(2)}_{21} w^{(3)}_{12} w^{(4)}_{11}
    + w^{(1)}_{22} w^{(2)}_{22} w^{(3)}_{12} w^{(4)}_{11}
    + w^{(1)}_{12} w^{(2)}_{21} w^{(3)}_{22} w^{(4)}_{12}
    \\
    w^{(1)}_{11} w^{(2)}_{21} w^{(3)}_{12} w^{(4)}_{21}
    + w^{(1)}_{21} w^{(2)}_{22} w^{(3)}_{12} w^{(4)}_{21}
    + w^{(1)}_{11} w^{(2)}_{21} w^{(3)}_{22} w^{(4)}_{22}
    &
    w^{(1)}_{12} w^{(2)}_{21} w^{(3)}_{12} w^{(4)}_{21}
    + w^{(1)}_{22} w^{(2)}_{22} w^{(3)}_{12} w^{(4)}_{21}
    + w^{(1)}_{12} w^{(2)}_{21} w^{(3)}_{22} w^{(4)}_{22}
\end{bmatrix}
$}.
\vspace{0.5em}

\noindent This path network has rank 2, even though all three paths pass through the same neuron in the middle layer. Thus, we have $r_a = 2$. Similarly, one can verify that $r_b = s = 1$. Note that $1 = r_1 \neq r_a + s = 3$.

\begin{figure}[t]
\centering
\begin{tabular}{cc}
    \begin{minipage}{0.65\textwidth} 
        \centering
        \begin{tikzpicture}[scale=0.9] 
            \node at (0, 3) {$n_0$};
            \node at (2, 3) {$n_1$};
            \node at (4, 3) {$n_2$};
            \node at (6, 3) {$n_3$};
            \node at (8, 3) {$n_4$};
            
            \foreach \y in {2, 1} {
                \fill (0, \y) circle (3pt);
                \fill (2, \y) circle (3pt);
                \fill (4, \y) circle (3pt);
                \fill (6, \y) circle (3pt);
                \fill (8, \y) circle (3pt);
            }

            \draw[magenta, thick] 
                (1, 0.5) -- (7, 0.5) -- (7, 2.5) -- (5, 2.5) -- (5, 1.5) -- (3, 1.5) -- (3, 2.5) -- (1, 2.5) -- cycle;

            \draw[orange, thick] 
                (1.2, 0.7) -- (6.8, 0.7) -- (6.8, 1.3) -- (1.2, 1.3) -- cycle;

            \draw[cyan, thick] 
                (1.1, 0.6) -- (6.9, 0.6) -- (6.9, 1.4) -- (5.1, 1.4) -- (5.1, 2.4) -- (2.9, 2.4) -- (2.9, 1.4) --(1.1, 1.4) -- cycle;

            \node at (0, 0) {\text{input}};
            \node at (2, 0) {\text{hidden}};
            \node at (4, 0) {\text{hidden}};
            \node at (6, 0) {\text{hidden}};
            \node at (8, 0) {\text{output}};
        \end{tikzpicture}
    \end{minipage} &
    \begin{minipage}{0.3\textwidth} 
        {\small   
        \vspace{-1em}
        \begin{tabular}{l}
             $\textcolor{magenta}{R_1} = \{(1,2,1),(2,2,1)$,\\
             $\phantom{R_2 = {}\;\;} (1,2,2),(2,2,2)\}$\\
             $\textcolor{cyan}{R_2} = \{(2,1,2), (2,2,2)\}$\\
            $\textcolor{orange}{S}\;\;= \{(2,2,2)\}$
\end{tabular} 
        }
    \end{minipage}
\end{tabular}
\caption{Deep ReLU network with $n_\ell = 2$ for $\ell=0,\ldots,4$ and two blocks given by the patterns $A_1 = [(1,1),(0,1),(1,1)]$ and $A_2 = [(0,1),(1,1),(0,1)]$.}
  \label{fig:deep-networks}
\end{figure}
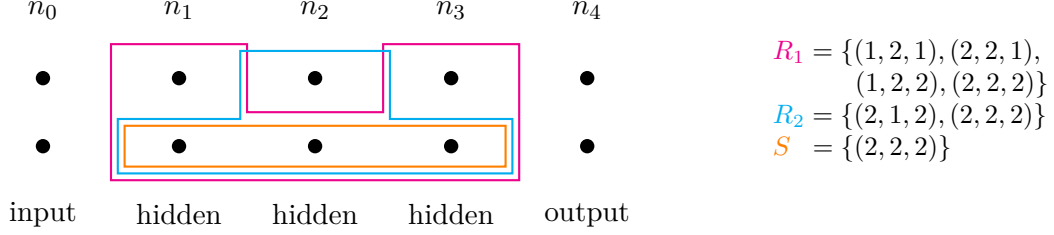
\end{example}

\begin{theorem}\label{thm:deep-invariants}
Let $n_{\min} = \min_{\ell\in[L-1]}n_\ell$, and let $I_{\min} = \{\ell \in [L-1] : n_{\ell} = n_{\min}\}$. Define $\ell_{\min}^+ = \max I_{\min}$ and $\ell_{\min}^- = \min I_{\min}$. The ideal $J^\mathbf{A}$ contains:

\begin{enumerate}
\item[1.] $(r_1+1)$-minors of $M_1$;
\item[2.] $(r_2+1)$-minors of $M_2$;

\item[3a.] $(n_{\min}+1)$-minors of $[M_1 \;|\; M_2]\;\;$ if $A_1^\ell = A_2^\ell$ for all $L-1 \geq \ell > \ell_{\min}^+$; 

\item[3b.] $(n_{\min}+1)$-minors of $[M_1^\top \;|\; M_2^\top]$ if $A_1^\ell = A_2^\ell$ for all $1 \leq \ell < \ell_{\min}^-$; 

\item[4.] $(t + 1)$-minors of $M_1 - M_2$.
\end{enumerate}
\end{theorem}

\begin{proof}
The proof that the polynomials of types 1 and 2 are in the ideal $J^\mathbf{A}$ is identical to that in Theorem \ref{thm:shallow-invariants}.

To show that polynomials of type 3a belong to the ideal, assume first that $\ell_{\min} = L-1$. Consider a submatrix $M'(\theta)$ of $[M_1(\theta) \; | \; M_2(\theta)]$ corresponding to an $(n_{\min} + 1)$-minor. By the multilinearity of the determinant, this minor can be expressed as a linear combination of determinants of matrices whose columns have the form $w^{(1)}_{p_1j} w^{(2)}_{p_2p_1} \cdots w^{(L-1)}_{p_{L-1}p_{L-2}}\cdot [w^{(L)}_{i_1 p_{L-1}} w^{(L)}_{i_2 p_{L-1}}\ldots w^{(L)}_{i_{n_{\min} +1} p_{L-1}}]^\top$ for some path $p$. 
Since the last hidden layer has only $n_{\min}$ neurons, there must exist two distinct columns in this determinant corresponding to two distict paths $p$ and $q$ where $p_{L-1} = q_{L-1}$. 
Consequently, each of these determinants has rank 1 and vanishes under the parametrization, implying that the original $(n_{\min} + 1)$-minor must also vanish. 
Now, suppose that $\ell_{\min} < L-1$ and that $A_1^\ell = A_2^\ell$ for all $\ell > \ell_{\min}$. Then any submatrix corresponding to an $(n_{\min} + 1)$-minor of $[M_1 \;|\; M_2]$ can be written as $H M'$, where $H$ is some matrix and $M'$ is a submatrix of $[M_1 \;|\; M_2]$ as in the $\ell_{\min} = L-1$ case. 
Since $\det(M') \in J^\mathbf{A}$, it follows that $\det(H M') \in J^\mathbf{A}$.

To show that polynomials of type 3b belong to the ideal, we similarly begin by assuming  $n_1 = n_{\min}$. Any $(n_{\min}+1)$-minor in $[M_1^\top(\theta) \; | \; M_2^\top(\theta)]$ can be expressed as a linear combination of determinants of matrices whose columns have the form $w^{(2)}_{p_2p_1} \cdots w^{(L-1)}_{p_{L-1}p_{L-2}}w^{(L)}_{i p_{L-1}}\cdot [
w^{(1)}_{p_1j_1} w^{(1)}_{p_1j_2} \ldots w^{(1)}_{p_1j_{n_{\min}+1}}]^\top$ for some path $p$. The remainder of the proof follows analogously.

Finally, to show that the polynomials of type 4 are in the ideal, we re-write $[M_1(\theta)\;|\;M_2(\theta)]$ as $[A(\theta)+C(\theta) \;|\; B(\theta)+C(\theta)]$, where $A(\theta)$ is the matrix parametrizing the path model corresponding to $R_1\setminus S$, $B(\theta)$ is the matrix parametrizing the path model corresponding to $R_2\setminus S$, and $C(\theta)$ is the matrix parametrizing the fully connected model induced by $S$. By our notation, $\rank A(\theta) \leq r_a$, $\rank B(\theta) \leq r_b$, and $\rank C(\theta) \leq s$ for any choice of $\theta$. Hence, $M_1(\theta) - M_2(\theta) = A(\theta) - B(\theta)$ has rank at most $t = r_a + r_b$. We conclude that all the polynomials of type 4 are in the ideal $J^\mathbf{A}$.
\end{proof}

\begin{example}
Consider the network in Example \ref{ex:deep-not-fully-connected}. Its ideal is generated by two quadratic determinants of $M_1$ and $M_2$. Since $n_{\min} = 2$ and $t = 2 + 1 = 3$, there are no generators of types 3 and~4. By computing the maximum rank of the Jacobian, it can be verified that there are no other generators of $J^\mathbf{A}$. This ideal has dimension 6 in $\CC^8$.
\end{example}

\begin{example}\label{ex:deep-example-2}
Consider a deep ReLU network with $L=4$ and $n_0=n_1=n_3=n_4=3$ and $n_2=2$. Let $A_1 = [(1,1,1), (1,1), (1,1,1))]$ and $A_2 = [(0,1,1), (1,1), (1,1,1))]$. In this case, $n_{\min} = 2$. Moreover, $r_a = 1, r_b = 0,$ and $s = 2$, so $t = 1$. The ideal $J^{\mathbf{A}}$ is generated by 9 quadrics of type 4 and all 3-minors of $[M_1\;|\;M_2]$, which correspond to the polynomials of types 1, 2, and 3a. There are no other generators, which was confirmed by computing\footnote{\url{https://github.com/yuliaalexandr/constraining-the-outputs-of-relu-networks}} the maximum rank of the Jacobian and comparing it to the dimension of the ideal generated by the polynomials in Theorem \ref{thm:deep-invariants}. The ideal $J^{\mathbf{A}}$ has dimension 12 in $\CC^{18}$. 
\end{example}

\begin{remark}
    While the value of $t$ is an upper bound on the rank of $M_1 - M_2$, it may not be tight even for generic data. For instance, consider a network with $L = 4, n_0 = n_4 = 4$ and $n_i = 3$ for $i = 1,2,3$ and the patterns $A_1 = [(1, 1, 0), (0, 1, 1), (1, 1, 0)]$ and $A_2 = [(0, 1, 1), (1, 1, 0), (0, 1, 1)]$. In this example, $r_a = r_b = 2$, so $t=4$. However, the generic rank of $M_1(\theta) - M_2(\theta)$ is~3, so the $4$-minors of $M_1 - M_2$ are in $J^\mathbf{A}$.
\end{remark}

\section{ReLU pattern varieties for multiple blocks}\label{sec:mult-blocks}

In this section, we investigate the polynomial invariants generating the ideal $J^\mathbf{A}$ in the case when~$\mathbf{A}$ contains $k$ distinct activation patterns: $A_1,\ldots, A_k$. The ReLU output variety in this case is parametrized as $\theta\mapsto [M_1(\theta),\ldots, M_k(\theta)]$. As before, let $r_i = \rank(M_i(\theta))$ for generic $\theta$.

\subsection{Shallow networks}
\label{sec:mult-blocks-shallow} 

\paragraph{Linear combinations.} 
In the case of shallow networks, each of the parametrized matrices $M_{i}(\theta) = W^{(2)} \diag(A_i) W^{(1)}$ can be written as a sum of rank-one components. 
That is, $M_{i}(\theta) = \sum_{j\in\operatorname{supp}A_i} N_j(\theta)$, where $\rank N_j(\theta) = 1$ for all $j\in \operatorname{supp}A_i$. Hence, for any $\lambda\in \mathbb{Z}^{k}$, we have 
$\sum_i \lambda_i M_{i}(\theta) = \sum_i \lambda_i \sum_{j} A_{ij} N_j(\theta)$. The nonzero coefficients of rank one components are the nonzero entries of $\sum_i \lambda_i A_i\in\mathbb{Z}^{n_1}$.  Therefore, we obtain
\begin{align}
\label{eq:rank-constraints-lin-comb}
    \rank( \sum_{i\in[k]} \lambda_i M_{i}(\theta)) \leq  |\operatorname{supp} (\sum_{i\in[k]} \lambda_i A_i) | 
\end{align}
for any $\lambda\in\mathbb{Z}^{k}$. These rank conditions give rise to polynomial constraints defined by minors. 

\begin{theorem}\label{thm:lin-combs-minors}
    For any $\lambda\in\mathbb{Z}^k$, the ideal $J^{\mathbf{A}}$ contains $(|\operatorname{supp} (\sum_{i\in[k]} \lambda_i A_i)| + 1)$-minors of the matrix $\sum_{i\in[k]} \lambda_i M_{i}$.
\end{theorem}

Theorem~\ref{thm:lin-combs-minors} suggests the presence of infinitely many minors in the ideal $J^\mathbf{A}$. However, not all of them are necessary to generate the ideal, as the Hilbert Basis Theorem ensures that every ideal admits a finite generating set. In fact, only those minors corresponding to linear combinations $\sum_{i \in [k]} \lambda_i A_i$ with minimal support can contribute to the minimal generators of $J^{\mathbf{A}}$. This motivates the following problem.

\begin{problem}\label{prob:min-gens}
    Which finite set of elements $\lambda \in \mathbb{Z}^k$ suffices to produce all minimal generators of~$J^\mathbf{A}$ that arise from the rank constraints in~\eqref{eq:rank-constraints-lin-comb}?
\end{problem}

Since only the support of the linear combination matters, it suffices to consider $\lambda$ up to scaling. Nevertheless, the problem remains nontrivial even after this reduction: we seek the sparsest possible representation of $\sum_{i=1}^k \lambda_i A_i$. Viewing the $A_i$ as columns of a matrix $\mathcal{A} = [A_1, A_2, \ldots, A_k]$, this becomes a problem of finding a vector $\lambda$ such that $\mathcal{A}\lambda$ has minimal support. This can be seen as a generalization of computing a minimal set of moves generating an ideal. In contrast to the setting of Markov bases \cite[Chapter 9]{sullivant2023algebraic}, we do not require that $\mathcal{A}\lambda = 0$, but instead aim to minimize the support of $\mathcal{A}\lambda$ among all such integer combinations.

\paragraph{Blocks matrices of linear combinations.}
We now consider a collection of linear combinations of matrices $M_i(\theta)$, i.e.,
\begin{align}\label{eq:T-blocks}
T_1(\theta) := \sum_{i=1}^k \lambda_i^{(1)} M_{i}(\theta), \quad \ldots, \quad T_n(\theta) := \sum_{i=1}^k \lambda_i^{(n)} M_{i}(\theta),
\end{align}
where $\lambda^{(j)}\in \mathbb{Z}^{k}$ for all $j\in[n]$. Let $v_j := \sum_{i=1}^k \lambda_i^{(j)} A_i$ and $t_j := \operatorname{supp}(v_j)$ for each $j\in[n]$. From~\eqref{eq:rank-constraints-lin-comb}, we have $\operatorname{rank}(T_j) \leq |t_j|$. Now we construct a block matrix $T(\theta)$ as follows:
$$T(\theta) = \begin{bmatrix}
    T_{i_{11}}(\theta) & \ldots & T_{i_{1\beta}}(\theta)\\
    \vdots & \ddots & \vdots\\
    T_{i_{\alpha1}}(\theta) & \ldots & T_{i_{\alpha\beta}}(\theta)\\
\end{bmatrix}  \text{ with } T_{i_{j_1k_1}}\neq T_{i_{j_2k_2}} \text{ for all } j_1 \neq j_2\text{ and } k_1 \neq k_2.$$
Note that each $T_j(\theta)$ admits a decomposition of the form $
T_j(\theta) = U(\theta) + U_j(\theta),$
where $U(\theta)$ is the maximal component common to all $T_j(\theta)$ for $j\in[n]$, and each $U_j(\theta)$ is specific to $T_j(\theta)$. Moreover, for generic $\theta$, we have $\operatorname{rank}U(\theta) = \left| \bigcap_j t_j \right|$. Then we may write $T(\theta)$ as:
\[
T(\theta) = U(\theta) \cdot 1_{\alpha \times \beta} + 
\begin{bmatrix}
U_{i_{11}}(\theta) & \cdots & U_{i_{1\beta}}(\theta) \\
\vdots & \ddots & \vdots \\
U_{i_{\alpha1}}(\theta) & \cdots & U_{i_{\alpha\beta}}(\theta) \\
\end{bmatrix},
\]
where $1_{\alpha \times \beta}$ is a $\alpha\times\beta$ matrix where each entry is equal to one, and each $U_{i_{jk}}(\theta)$ is one of the matrices $U_1(\theta), \ldots, U_n(\theta)$. The block matrix above can be written as a sum of $n$ structured block matrices $\sum_{j=1}^n B_j(\theta)$. Each $B_j(\theta)$ is a block matrix, whose nonzero blocks are all equal to $U_j(\theta)$. Moreover, such nonzero blocks only appear within a single block row and/or a single block column. Thus:
\[
\operatorname{rank}B_j(\theta) \leq \operatorname{rank}U_j (\theta) = \left| t_{j} \setminus \cap_{j'\in[n]} t_{j'}\right|.
\]

\medskip

\noindent Putting everything together, the total rank of $T(\theta)$ is bounded by
\[
\operatorname{rank}T(\theta) \leq \left| \cap_{j\in[n]} t_j \right| + \sum_{j=1}^n \left| t_j \setminus \cap_{j'\in[n]} t_{j'} \right|.
\]

For any $\lambda^{(1)},\ldots, \lambda^{(n)}\in\mathbb{Z}^k$ and any $T(\theta)$ as above, let $t:= \left| \cap_j t_j \right| + \sum_{j=1}^n \left| t_j \setminus \cap_{j'} t_{j'} \right|$. Let $T$ be the matrix of indeterminates corresponding to $\eqref{eq:T-blocks}$. We obtain the following theorem.
\begin{theorem}\label{thm:minors-lin-comb-blocks}
The ideal $J^\mathbf{A}$ contains all $(t+1)$-minors of $T$.
\end{theorem}

\begin{example}
    Let $n_0 = n_2 = 3$ and let $n_1 = 4$. Consider the pattern $\mathbf{A}$ with three blocks, determined by $A_1 = [1,1,0,0], A_2 = [1,0,1,0], A_3=[1,0,0,1]$. The ideal $J^\mathbf{A}$ is generated by:
    \begin{enumerate}
        \item 3 cubics, given by the determinants of $M_1$, $M_2$, and $M_3$, corresponding to $\lambda = (1,0,0)$, $(0,1,0)$, and $(0,0,1)$, respectively;
        \item 3 cubics, given by the determinants of $M_1 - M_2$, $M_1 - M_3$, and $M_2 - M_3$, corresponding to $\lambda = (1,-1,0)$, $(1,0,-1)$, and $(0,1,-1)$, respectively;
        \item 45 quintics, given by algebraically independent 5-minors of
        \begin{align}\label{eq:quintics}
            \begin{bmatrix}
            M_1 & M_2 \\ 
            M_3 & M_2
        \end{bmatrix}, 
        \begin{bmatrix}
            M_1 & M_2 \\ 
            M_3 & M_3
        \end{bmatrix}, 
        \begin{bmatrix}
            M_2 & M_3 \\ 
            M_1 & M_1
        \end{bmatrix},
        \begin{bmatrix}
            M_2 & M_3 \\ 
            M_1 & M_3
        \end{bmatrix},
        \begin{bmatrix}
            M_3 & M_1 \\ 
            M_2 & M_2
        \end{bmatrix},
        \begin{bmatrix}
            M_3 & M_1 \\ 
            M_2 & M_1
        \end{bmatrix}.
        \end{align}
    \end{enumerate}
    
The first matrix in \eqref{eq:quintics} is the block matrix $T$, whose blocks are given by
\[
T = \begin{bmatrix}
T_{11} & T_{12} \\
T_{21} & T_{22}
\end{bmatrix}
\]
with $T_{11} = T_1$, $T_{12} = T_{22} = T_2$, and $T_{21} = T_3$. Here, $T_1 = M_1$ (associated to $\lambda^{(1)} = (1,0,0)$), $T_2 = M_2$ (associated to $\lambda^{(2)} = (0,1,0)$), and $T_3 = M_3$ (associated to $\lambda^{(3)} = (0,0,1)$).

There are no other generators\footnote{\url{https://github.com/yuliaalexandr/constraining-the-outputs-of-relu-networks}}, as we verified in \texttt{Macaulay2} \cite{M2}. The ideal $J^\mathbf{A}$ has co-dimension~$7$.
\end{example}

\begin{example}
    Let $n_0 = n_1 = n_2 = 4$ and consider again the three blocks, given by $A_1 = [1,1,0,0], A_2 = [1,0,1,0], A_3=[1,0,0,1]$. The ideal $J^\mathbf{A}$ contains:
    \begin{enumerate}
        \item 48 cubics, given by the 3-minors of $M_1$, $M_2$, and $M_3$;
        \item 48 cubics, given by the 3-minors of $M_1-M_2$, $M_2-M_3$, and $M_2-M_3$;
        \item 120 quartics, given by the 4-minors of $[M_i \;|\; M_j]$ and $[M_i^\top \;|\; M_j^\top]$ for $i\neq j$;
        \item 40 quartics, given by the 4-minors of $[M_1 - M_2 \;|\; M_2 - M_3]$ and $[(M_1 - M_2)^\top \;|\; (M_2 - M_3)^\top]$;
        \item 2000 quintics, given by algebraically independent 5-minors of \eqref{eq:quintics}.
       \end{enumerate} 
We cannot determine whether the above generators are sufficient, as the computation of the dimension of the ideal they generate did not terminate. However, the ideal of the pattern variety $I_X^{\mathbf{A}}$ has dimension 20 in $\mathbb{C}^{27}$, which can be confirmed through the Jacobian rank computation.
\end{example}

\subsection{Deep networks}
For any activation pattern $A\in \{0,1\}^N$, recall that $P_A$ denotes the set of all $A$-active paths. A fixed ReLU network with hidden layer widths $n_1, \dots, n_{L-1}$ admits a total of $q = n_1 n_2 \cdots n_{L-1}$ distinct paths. Fix a lexicographic ordering of all such paths and index them accordingly. Then, for each activation pattern $A$, define the vector $Q_A \in \{0,1\}^q$ by setting $(Q_A)_j = 1$ if the $j$th path (in the fixed ordering) belongs to $P_A$, and $(Q_A)_j = 0$ otherwise.

We generalize the theory in Section~\ref{sec:mult-blocks-shallow} to deep networks by decomposing each matrix $M_i(\theta)$ into a sum of rank-one matrices, each corresponding to an $A_i$-active path. Indeed, let $M_i(\theta) = \sum_{p\in P_{A_i}} N_p(\theta)$, where each $N_p(\theta)$ is a rank-one matrix determined by the $A_i$-active path $p$. Then for any $\lambda\in \mathbb{Z}^{k}$, we have $\sum_{i\in[k]}\lambda_iM_i(\theta) = \sum_{i\in[k]}\lambda_i\sum_{p\in P_A}N_p(\theta)$. 
The non-zero coefficients of the rank-one matrices on the right hand-side are the nonzero entries of $\sum_{i\in [k]}\lambda_i Q_{A_i}$. So we obtain the inequality for any $\lambda\in\mathbb{Z}^k$:
\begin{align}\label{eq:linear-comb-deep}
    \rank (\sum_{i\in[k]} \lambda_i M_i(\theta)) \leq |\operatorname{supp}(\sum_{i\in [k]} \lambda_i Q_{A_{i}})|,
\end{align}
resulting in minor equations in $J^\mathbf{A}$.

\begin{remark}
    The inequality in \eqref{eq:linear-comb-deep} is not usually tight, even for generic $\theta$. 
    For instance, consider Example \ref{ex:deep-example-2}, where $\rank(M_1(\theta) - M_2(\theta))$ is 1 for generic $\theta$, but it is a sum of 6 rank-one matrices corresponding to paths. So, \eqref{eq:linear-comb-deep} would only bound it by 6. 
\end{remark} 

One can generalize the block matrix results for linear combinations in Section~\ref{sec:mult-blocks-shallow} to deep networks by focusing on common paths rather than individual neurons. However, the resulting rank bounds are quite loose.

\section{Dimensions of ReLU pattern varieties for shallow networks}\label{sec:dim-shallow}

In this section, we study the dimension of a shallow ReLU network without biases like in Section~\ref{sec:shallow}. The \textit{dimension} of the function space of a ReLU network for a particular activation pattern $\mathbf{A}$ is defined to be the maximum rank achieved by the Jacobian of the parametrization map at the locations where it is differentiable. This is also the dimension of the ideal $J^{\mathbf{A}}$. 

\subsection{Two blocks}
Let $\Jac$ denote the Jacobian matrix of the parametrization $\varphi^\mathbf{A}:\theta\mapsto[M_1(\theta)\;|\;M_2(\theta)]$. Arrange the rows of $\Jac$ into three blocks, each corresponding to a subset of parameters:
\begin{align*}
&\{w^{(1)}_{i1}, \ldots, w^{(1)}_{in_0}, w^{(2)}_{1i}, \ldots w^{(2)}_{n_2i}: i\in U\},
\quad \text{where $U \in \{R_1\setminus S,\;\;\; S,\;\;\; R_2\setminus S\}$}.
\end{align*}
Arrange the columns of $\Jac$ into two blocks: one labeled by the variables $m^{(1)}_{ij}$ in $M_1$ and another labeled by the variables $m^{(2)}_{ij}$ in $M_2$, both ordered column by column. Then $\Jac$ is a $n_1(n_0+n_2)\times 2n_0n_2$ block matrix of the following form:
\begin{equation}\label{eq:jac-blocks}
\begin{matrix}
    (r_1-s)(n_0+n_2)\;\{\\
    \;\;\;\;\;\;\;\;\;\;\,k(n_0+n_2)\;\{\\ 
    (r_2-s)(n_0+n_2)\;\{
\end{matrix}
\begin{bmatrix}
    D_a & 0\\
    D_{c} & D_{c}\\
    0 & D_b
\end{bmatrix}.
\end{equation}
Each of the matrices $D_a, D_b, D_c$ are stacks of smaller $(n_0+n_2)\times n_0 n_2$ matrices. For example, 
$$ D_a =
\begin{bmatrix}
    E_1\\
    E_2\\
    \vdots\\
    E_{r_1-s}
\end{bmatrix}\text{ where } 
E_i = \left[\begin{smallmatrix} I_{n_0} \otimes [w^{(2)}_{1i} & w^{(2)}_{2i} & \ldots & w^{(2)}_{n_2i}] \\
[w^{(1)}_{i1} & w^{(1)}_{i2} & \ldots & w^{(1)}_{in_0}] \otimes I_{n_2}
\end{smallmatrix}\right]$$
for all $i\in[r_1 - s]$, where $\otimes$ denotes the Kronecker product of matrices.

Note that $D_a$ is the Jacobian of the parametrization of the full linear model with $r_1 - s$ active neurons. Hence, $\rank D_a = d_a := (r_1 - s)(n_0 + n_2) - (r_1 - s)^2$. Similarly, $\rank D_b = d_b := (r_2 - s)(n_0 + n_2) - (r_2 - s)^2$ and $\rank D_{c} = d_{c} := s(n_0 + n_2) - s^2$. The next proposition provides an upper and lower bound on $\rank\Jac$, and hence on $\dim J^\mathbf{A}$.

\begin{proposition}\label{prop:rank-jac} The dimension of the ideal $J^\mathbf{A}$ is bounded as follows:
$$\max\{r_1(n_0 + n_2) - r_1^2 + d_{2},\;\;\; r_2(n_0 + n_2) - r_2^2 + d_{a}\}  \leq \dim J^\mathbf{A} \leq \min\{d_a + d_b + d_{c},\; 2n_0n_2\}.$$
\end{proposition}

\begin{proof}
The upper bound follows directly from observing that $\rank[\;D_{c}\;\;D_{c}\;]=d_{c}$.
Note that $\rank [\;D_a\;\; D_{c}\;]^\top = r_1(n_0 + n_2) - r_1^2$ and $\rank [\;D_{c}\;\; D_b\;]^\top = r_2(n_0 + n_2) - r_2^2$, since $[\;D_a\;\; D_{c}\;]^\top$ and $[\;D_{c}\;\; D_b\;]^\top$ are the Jacobians of the parametrization of the full linear models with $r_1$ and $r_2$ active neurons, respectively. This proves the lower bound.
\end{proof}

The \textit{expected dimension} of the model is the upper bound $\min\{d_a+d_b+d_{c},\; 2n_0n_2\}$. For some models, the true dimension is less than the expected one, as demonstrated in the next example.

\begin{example}\label{ex:dim-drop}
Consider a shallow ReLU with $(n_0,n_1,n_2)=(3,4,3)$. Let $A_1 = [1,1,1,0]$ and $A_2=[0,1,1,1]$. From Theorem \ref{thm:shallow-invariants}, we know that there is a degree 3 generator of type 4 in the ideal~$J^\mathbf{A}$. In fact, this is the only generator of the ideal, so the corresponding variety is a hypersurface of dimension 17 in $\CC^{18}$. Yet, the expected dimension is $d_a + d_b + d_{c} = 18$.
\end{example}

Another way to see the expected dimension is by counting parameters. The rows of the Jacobian are divided into three blocks as in \eqref{eq:jac-blocks}, and the total number of parameters  is $n_1(n_0 + n_2)$. However, there is a symmetry in the parameters induced by the action of the group $G : = \text{GL}_{r_1-s}\times \text{GL}_{s} \times \text{GL}_{r_2-s}$. Indeed, $\varphi^\mathbf{A} (W^{(1)}, W^{(2)}) = \varphi^\mathbf{A} (Z W^{(1)}, W^{(2)}Z^{-1})$ where $Z$ is the block matrix $\diag(Z_a, Z_{c}, Z_{b})$ with $Z_a \in \text{GL}_{r_1-s}, Z_{c} \in \text{GL}_{s}, Z_b \in \text{GL}_{r_2-s}.$ So, the map $\varphi^\mathbf{A}$ is invariant under the action of $G$. The generic fiber of this map has dimension at least $(r_1-s)^2 + (r_2-s)^2 + s^2$, and so the dimension of $J^\mathbf{A}$ is at most $d_a + d_b + d_c$. The following theorem shows that if the width of the hidden layer is ``small'' enough compared to the widths of the input and the output layers, then the expected dimension is achieved.

\begin{theorem}\label{thm:expected-dim} If $n_0 \geq n_1$ and $n_1 \leq n_2$, the ideal $J^\mathbf{A}$ has the expected dimension $d_a + d_b + d_c$. \end{theorem}
\begin{proof}
In the shallow case, the map in \eqref{eq:multilinear-map-version-1} takes the form
\begin{align*}
\varphi&:\RR^{n_1 \times n_0} \times \RR^{n_2 \times n_1} \to \RR^{2(n_2\times n_0)}\\
& (W^{(1)}, W^{(2)})\mapsto (W^{(2)}_{[:,\{1,\ldots, r_1\}]}W^{(1)}_{[\{1,\ldots, r_1\}, :]}, W^{(2)}_{[:,  \{s+1,\ldots, n_1\}]}W^{(1)}_{[\{s+1,\ldots, n_1\}, :]}). 
\end{align*}

Suppose $\varphi(W^{(1)}, W^{(2)}) = \varphi(Y^{(1)}, Y^{(2)})$ for two generic parameter choices $(W^{(1)}, W^{(2)})$ and $(Y^{(1)}, Y^{(2)})$. We will show there exists a unique $Z \in \mathrm{GL}_{n_1}$ such that $Z W^{(1)} = Y^{(1)}$, and that $Z$ is block diagonal:
$$
Z = \mathrm{diag}(Z_a, Z_{c}, Z_b), \quad \text{with} \quad Z_a \in \mathrm{GL}_{r_1 - s}, \; Z_{c} \in \mathrm{GL}_s, \; Z_b \in \mathrm{GL}_{r_2-s}.
$$

Since $n_0 \geq n_1$, $W^{(1)}$ is generically surjective and admits a right inverse. Thus, there is a unique $Z\in\mathrm{GL}_{n_1}$ such that $Z W^{(1)} = Y^{(1)}$.
 Let 
$$Z = \begin{pmatrix} Z_{a} & Z_{ac} & Z_{ab} \\
Z_{ca} & Z_{c} & Z_{cb} \\
Z_{ba} & Z_{bc} & Z_{b} \end{pmatrix}\;\text{where}\;\;
\begin{aligned}
Z_a &\in \RR^{(r_1-s)\times (r_1 - s)}, Z_{ac} \in \RR^{(r_1-s)\times s}, Z_{ab} \in \RR^{(r_1-s)\times (r_2-s)}\\
Z_{ca} &\in \RR^{s\times (r_1-s)}, Z_c \in  \RR^{s\times s}, Z_{cb} \in \RR^{s\times (r_2-s)},\\
Z_{ba} &\in \RR^{(r_2-s)\times (r_1-s)}, Z_{bc} \in \RR^{(r_2-s)\times s}, Z_b \in \RR^{(r_2-s)\times (r_2 - s)}.\\
\end{aligned}$$

Decompose $W^{(1)}$ and $W^{(2)}$ according to the overlapping structure:

$$
W^{(1)} = 
\begin{pmatrix}
W^{(1)}_a \\
W^{(1)}_c \\
W^{(1)}_b
\end{pmatrix},
\quad
\text{where}
\quad
\begin{aligned}
W^{(1)}_a &\in \mathbb{R}^{(r_1 - s) \times n_0}, \\
W^{(1)}_c &\in \mathbb{R}^{s \times n_0}, \\
W^{(1)}_b &\in \mathbb{R}^{(r_2 - s) \times n_0}.
\end{aligned}
$$

and 

$$\begin{pmatrix}
W^{(2)}_a & W^{(2)}_c & W^{(2)}_b
\end{pmatrix},
\quad
\text{where}
\quad
\begin{aligned}
W^{(2)}_a &\in \mathbb{R}^{n_2 \times (r_1 - s) }, \\
W^{(2)}_c &\in \mathbb{R}^{n_2 \times s}, \\
W^{(2)}_b &\in \mathbb{R}^{n_2 \times (r_2- s)}.
\end{aligned}$$

\noindent Similarly, let $Y^{(1)} = (Y^{(1)}_a, Y^{(1)}_c, Y^{(1)}_b)^\top$ and $Y^{(2)} = (Y^{(2)}_a, Y^{(2)}_c, Y^{(2)}_b)$. Since $ZW^{(1)} = Y^{(1)}$, we obtain:
$$ZW^{(1)} = \begin{pmatrix}
    Z_{a}W^{(1)}_a + Z_{ac}W^{(1)}_c + Z_{ab}W^{(1)}_b\\
    Z_{ca}W^{(1)}_a + Z_{c}W^{(1)}_c + Z_{cb}W^{(1)}_b\\
    Z_{ba}W^{(1)}_a + Z_{bc}W^{(1)}_c + Z_{b}W^{(1)}_b\\
\end{pmatrix} = \begin{pmatrix}
    Y_a^{(1)} \\
    Y_c^{(1)} \\
    Y_b^{(1)}
\end{pmatrix} = Y^{(1)}.$$
The equality $W^{(2)}_{[:,\{1,\ldots, r_1\}]}W^{(1)}_{[\{1,\ldots, r_1\}, :]} = Y^{(2)}_{[:,\{1,\ldots, r_1\}]}Y^{(1)}_{[\{1,\ldots, r_1\}, :]}$,~gives
\begin{align*}
W_a^{(2)}W_a^{(1)} &+ W_c^{(2)}W_c^{(1)} = Y_a^{(2)}Y_a^{(1)} + Y_c^{(2)}Y_c^{(1)} = \\
&= (Y_a^{(2)}Z_a + Y_c^{(2)}Z_{ca})W_a^{(1)} + (Y_a^{(2)}Z_{ac} + Y_c^{(2)}Z_{c})W_c^{(1)} + (Y_a^{(2)}Z_{ab} + Y_c^{(2)}Z_{cb})W_b^{(1)}.
\end{align*}

\noindent Similarly, $W^{(2)}_{[:,  \{k+1,\ldots, n_1\}]}W^{(1)}_{[\{k+1,\ldots, n_1\}, :]}) = Y^{(2)}_{[:,  \{k+1,\ldots, n_1\}]}Y^{(1)}_{[\{k+1,\ldots, n_1\}, :]}$ gives

\begin{align*}
W_c^{(2)}W_c^{(1)} &+ W_b^{(2)}W_b^{(1)} = Y_c^{(2)}Y_c^{(1)} + Y_b^{(2)}Y_b^{(1)} \\
&= (Y_c^{(2)}Z_{ca} + Y_b^{(2)}Z_{ba})W_a^{(1)} + (Y_c^{(2)}Z_{c} + Y_b^{(2)}Z_{bc})W_c^{(1)} + (Y_c^{(2)}Z_{cb} + Y_b^{(2)}Z_{b})W_b^{(1)}.
\end{align*}
\noindent Since $n_0 \geq n_1$, the rows of $W_a^{(1)}, W_c^{(1)}, W_b^{(1)}$ are linearly independent. Hence:
\begin{enumerate}
\item[(1)] $Y_a^{(2)}Z_{ab} + Y_c^{(2)} Z_{cb} = 0$,
\item[(2)] $Y_c^{(2)}Z_{ca} + Y_b^{(2)}Z_{ba} = 0,$
\item[(3)]  $Y_a^{(2)}Z_{ac} + Y_c^{(2)}Z_{c} = W^{(2)}_c$,
\item[(4)]  $Y_c^{(2)}Z_{c} + Y_b^{(2)}Z_{bc} = W^{(2)}_c$.
\end{enumerate} 

Consider the equation (1). If $r_2 - s = 0$, then the blocks $Z_{ab}$ and $Z_{cb}$ do not appear in~$Z$. So, we may assume $r_2 - s \geq 1$. Since $n_1 \leq n_2$ by assumption, the columns of $Y_a^{(2)}$ and $Y_c^{(2)}$ are linearly independent, and thus $Z_{ab} = 0$ and $Z_{cb} = 0$. Similarly, from equation (2), we deduce that $Z_{ca} = 0$ and $Z_{ba} = 0$. Finally, from equations (3) and (4), we obtain
$Y^{(2)}_a Z_{ac} = Y^{(2)}_b Z_{bc}$. Since $n_1 \leq n_2$, the matrix $[Y_a^{(2)} \; Y_b^{(2)}]$ is injective, so its columns are linearly independent. Therefore, $Z_{ac} = 0$ and $Z_{ba} = 0$, as desired. This proves that a generic fiber of this map has no symmetries other than those arising from the action of $\text{GL}_{r_1 - s} \times \text{GL}_{s} \times \text{GL}_{r_2 - s}$, and hence the ideal $J^{\mathbf{A}}$ has the expected dimension.
\end{proof}

In fact, we get the following corollary to the proof above.
\begin{corollary}
Let $s = \sum_{j=1}^{n_1}A_{1j}A_{2j}$ be the number of neurons in which two blocks intersect. If either of the following conditions holds: 
\begin{itemize} 
\item $n_0 \geq n_1$ and $n_1 \leq n_2 + s$,
\item $n_2 \geq n_1$ and $n_1 \leq n_0 + s$, \end{itemize} 
then the ideal $J^\mathbf{A}$ has the expected dimension $d_a + d_b + d_c$. 
\end{corollary}
\begin{proof}
For the case when $n_0\geq n_1$, the proof of Theorem \ref{thm:expected-dim} goes through: since $n_1 - s \leq n_2$, the matrix $[Y_a^{(2)} \; Y_b^{(2)}]$ is still injective.

The case when $n_2 \geq n_1$ and $n_1 \leq n_0 + s$ is similar: $W^{(2)}$ is generically injective and admits a left inverse. Thus, there is a unique $Z\in\mathrm{GL}_{n_1}$ such that $W^{(2)}Z^{-1} = Y^{(2)}$. The rest of the proof carries out analogously.

\end{proof}

\subsection{Multiple blocks}
We again let $\Jac$ denote the Jacobian matrix of the parametrization $\varphi^\mathbf{A}:\theta\mapsto[M_1(\theta)\;|\;\ldots \;|\; M_k(\theta)]$. 

\begin{example}\label{ex:dim-mult-blocks}
Let $n_0 = 4, n_1 = n_2 = 5$. Consider the shallow network, defined by three blocks with patterns $A_1 = [1,0,0,1,1], A_2 = [0,1,0,1,1], A_3 = [0,0,1,0,1]$. The dimension of the ideal $J^\mathbf{A}$ is 40 in $\CC^{60}$. The Jacobian can be written in the block form as follows
    $$\Jac = \begin{bmatrix}
        D_{1} & 0 & 0 \\
        0 & D_{2} & 0\\
        0 & 0 & D_{3}\\
        D_{12} & D_{12}  & 0\\
        D_{123} & D_{123} & D_{123}\\
    \end{bmatrix},$$
where $D_1$, $D_2$, and $D_3$ are the Jacobian matrices of the full linear models induced separately by the first three hidden neurons, and $D_{12}$ and $D_{123}$ are the Jacobian matrices of the models induced by the fourth and fifth hidden neurons, respectively. Each block has rank 8. The total rank of the $\Jac$ is $5\cdot 8 = 40$, which is the expected dimension.
\end{example}

To generalize the dimension results to the case of multiple blocks, we begin by introducing some notation. Fix an ordering $1, \ldots, n_1$ of the neurons in the hidden layer. As before, consider a network with $k$ blocks, and let $B_i\subseteq [n_1]$ denote the set of all active neurons active in the $i$th block, for $i\in[k]$. For any $I\subseteq [k]$, let $C_I := \bigcap_{i\in I}B_i \setminus \bigcup_{i\notin I}B_j$. For instance, in Example \ref{ex:dim-mult-blocks}, $B_{1} = \{1,4,5\}, B_{2} = \{2,4,5\}$, and $C_{\{1,2\}} = \{4\}$. Let $G := \prod_{I\subseteq [k]}\operatorname{GL}_{|C_I|}$, where the product denotes the Cartesian product of general linear groups over $\CC$. There is a natural action of $G$ on the parameter space, which preserves the image of the model.

\begin{theorem}\label{thm:expected-dim-multiple-blocks}
    If $n_0 \geq n_1$ and $n_1 \leq n_2$, the ideal $J^\mathbf{A}$ has the expected dimension $n_1(n_0 + n_2) - \sum_{I\subseteq [k]}\operatorname{|C_I|^2}$.
\end{theorem}

\begin{proof}
When $k = 1$, there is a single block, so the ideal $J^{\mathbf{A}}$ corresponds to the full linear model. Its dimension is $n_1(n_0 + n_2) - n_1^2$, which agrees with the formula since $G = \operatorname{GL}_{n_1}$. The case $k = 2$ follows directly from Theorem~\ref{thm:expected-dim}. Now suppose $k > 2$, and choose any integer $0 < n < k$. Define:
\begin{align*}
&R_1 := \{I \subseteq [k]: C_I\neq \varnothing \text{ and } [n] \cap I \neq \varnothing \text{ and }([k]\setminus[n])\cap I = \varnothing \} \\
&R_2 := \{I \subseteq [k]: C_I\neq \varnothing \text{ and } [n] \cap I = \varnothing \text{ and }([k]\setminus[n])\cap I \neq \varnothing \} \\
&S \;\;:= \{I \subseteq [k]: C_I\neq \varnothing\text{ and } [n] \cap I \neq \varnothing \text{ and }([k]\setminus[n])\cap I \neq \varnothing \}  
\end{align*}
These three sets are pairwise disjoint and satisfy
$$R_1 \cup R_2 \cup S = \{I \subseteq [k] : C_I \neq \varnothing\} \quad \text{and} \quad |\cup_{I\in R_1 \cup R_2 \cup S}C_I| = n_1.$$

By the inductive hypothesis, the model induced by the neurons in $\bigcup_{I \in R_j} C_I$ has dimension $|R_j|(n_0 + n_2) - \sum_{I \in R_j} |C_I|^2$ for each $j = 1, 2$, and the model induced by the neurons in $\bigcup_{I \in S} C_I$ has dimension $|S|(n_0 + n_2) - \sum_{I \in S} |C_I|^2$. Treating this as a two-block model and applying Theorem~\ref{thm:expected-dim}, we find that the dimension of $J^{\mathbf{A}}$ is
\[
(|R_1| + |R_2| + |S|)(n_0 + n_2) - \sum_{I \in R_1 \cup R_2 \cup S} |C_I|^2 = n_1(n_0 + n_2) - \sum_{I \subseteq [k]} |C_I|^2,
\]
as desired.
\end{proof}

\begin{remark}
Note that the condition in Theorem~\ref{thm:expected-dim-multiple-blocks} is sufficient but not necessary. For instance, the network in Example~\ref{ex:dim-mult-blocks} satisfies $n_1 > n_0$, yet still achieves the expected dimension.
\end{remark}

\section{Conclusion and future work} 

We studied the relations or dependencies between the outputs of a feedforward ReLU network for a given input dataset, over different regions of the parameter space. 
Abstracting away the specific input points, we also studied the relations between the linear pieces of the functions computed by the network. We identified determinantal constraints associated with ReLU networks covering data within a single or multiple activation patterns, networks with and without biases, shallow and deep networks. 

We believe that the systematic study of the constraint on the outputs of a network that we have pursued here can be of interest for several applications. 
In particular, the framework could be used to better understand the functional constraints, especially when paired with constraints on the parameters, and thereby inform on the ability of ReLU networks to generalize outside of a training set. 
The identification of constraints on the outputs naturally is also of interest in the context of neural network verification, where one aims to constraint the possible outputs of a neural network, e.g., to show that for all inputs within a given neighborhood, the outputs will be within a specified neighborhood. Many of the existing approaches for neural network verification focus on fast linear bounds for individual neurons, which could be supplemented by polynomial constraints of the form that we have presented here. 
Our work is also related to the notion of feature manifolds and, in particular, the work of \cite{pokutta}, which studies the algebraic structure of neural network function spaces. By restricting to any subnetwork, we can recover the invariants of the corresponding \textit{feature variety}, providing a direct connection to their framework through the lens of vanishing ideals. 

Several open questions remain. 
In particular, resolving Conjecture \ref{conj:sufficiency} for shallow networks with two blocks would be especially significant, as it would ensure that all constraints are known in that case. 
Moreover, in the case of multiple blocks, it remains to determine which of the infinitely many rank conditions yield a minimal (or even finite) generating set for the ideal. This is the focus of Problem~\ref{prob:min-gens}. 
Another promising direction for future work is to extend the dimension results from Section~\ref{sec:dim-shallow} to deep networks. In particular, we have identified conditions under which shallow networks attain the expected dimension for subsets of activation regions. Determining the corresponding conditions in the deep case is a natural next step. 

Furthermore, in this work we have focused on the variety defining the image associated with ReLU networks. However, the actual image is a semialgebraic set, since the input space is constrained by natural inequalities determined by the activation patterns. Understanding how these inequalities translate into inequalities on the output space, thus restricting the set of possible outputs even further, is an important next step toward characterizing the functions that ReLU networks can realize. 
In this context, it will be valuable also to investigate in more depth the CPWL parameter variety following \cite{brandenburg2024the} and its connections to the output and pattern varieties. 
Yet another angle that could be considered to reason about the function space of ReLU networks is via softplus activations, which are smooth, and can be linked to other models studied in the literature, particularly the Restricted Boltzmann Machine, as discussed in \cite{10.1007/978-3-319-97798-0_4,sonthalia2023supermodularranksetfunction}. 

Our study also highlighted certain structures that can be of independent interest from an algebro-geometric perspective. 
The path variety introduced in Section~\ref{sec:deep-networks} merits further investigation, particularly in understanding the rank of the associated matrices. Notably, the overall rank can be strictly less than the number of paths, and even when all paths pass through a bottleneck, the rank can still exceed the width of the bottleneck. 
Finally, it would be interesting to also consider architectures where the input data is fed at multiple layers, where the outputs then have a similar structure in terms of the input data and the parameters. 

\subsection*{Acknowledgment} 

This project has been supported by DFG in SPP 2298 (FoDL) project 464109215. YA and GM have been supported by DARPA in AIQ project HR00112520014:P00002. GM has also been supported by NSF CCF-2212520, NSF DMS-2145630, and BMFTR in DAAD project 57616814 (SECAI). 
We thank Rishi Sonthalia for insightful discussions and valuable input during the early stages of this project.

\bibliographystyle{plain} 
\bibliography{references}

@misc{sonthalia2023supermodularranksetfunction,
      title={Supermodular Rank: Set Function Decomposition and Optimization}, 
      author={Rishi Sonthalia and Anna Seigal and Guido Mont\'ufar},
      year={2023},
      eprint={2305.14632},
      archivePrefix={arXiv},
      primaryClass={math.CO},
      url={https://arxiv.org/abs/2305.14632}, 
}

@article{tran2023minimal,
      title={Minimal Representations of Tropical Rational Functions}, 
      author={Ngoc Mai Tran and Jidong Wang},
      year={2023},
      journal={Algebraic Statistics}, 
volume = {15}, 
number = {1}, 
pages = {27--59}
}

@misc{arjevani2025geometry,
      title={Geometry and Optimization of Shallow Polynomial Networks}, 
      author={Yossi Arjevani and Joan Bruna and Joe Kileel and Elzbieta Polak and Matthew Trager},
      year={2025},
      eprint={2501.06074},
      archivePrefix={arXiv},
      primaryClass={cs.LG},
      url={https://arxiv.org/abs/2501.06074}, 
}

@article{10.1007/s10994-023-06355-4, 
author = {Bona-Pellissier, Joachim and Bachoc, Fran\c{c}ois and Malgouyres, Fran\c{c}ois}, title = {Parameter identifiability of a deep feedforward {R}e{LU} neural network}, year = {2023}, issue_date = {Nov 2023}, publisher = {Kluwer Academic Publishers}, 
noaddress = {USA}, 
volume = {112}, 
number = {11}, 
issn = {0885-6125}, 
url = {https://doi.org/10.1007/s10994-023-06355-4}, 
doi = {10.1007/s10994-023-06355-4}, journal = {Mach. Learn.}, month = aug, pages = {4431–4493}, numpages = {63} }

@inproceedings{bona-pellissier2022local,
title={Local Identifiability of Deep {R}e{LU} Neural Networks: the Theory},
author={Joachim Bona-Pellissier and Francois Malgouyres and Francois Bachoc},
booktitle={Advances in Neural Information Processing Systems},
noeditor={Alice H. Oh and Alekh Agarwal and Danielle Belgrave and Kyunghyun Cho},
year={2022},
url={https://openreview.net/forum?id=-3cHWtrbLYq}
}

@InProceedings{pmlr-v49-telgarsky16,
  title = 	 {benefits of depth in neural networks},
  author = 	 {Telgarsky, Matus},
  booktitle = 	 {29th Annual Conference on Learning Theory},
  pages = 	 {1517--1539},
  year = 	 {2016},
  noeditor = 	 {Feldman, Vitaly and Rakhlin, Alexander and Shamir, Ohad},
  volume = 	 {49},
  series = 	 {Proceedings of Machine Learning Research},
  noaddress = 	 {Columbia University, New York, New York, USA},
  month = 	 {23--26 Jun},
  publisher =    {PMLR},
  url = 	 {https://proceedings.mlr.press/v49/telgarsky16.html}
}

@article{embedding-ReLU-identifiability,
	Author = {Stock, Pierre and Gribonval, R{\'e}mi},
	Da = {2023/04/01},
	Doi = {10.1007/s00365-022-09578-1},
	Id = {Stock2023},
	Isbn = {1432-0940},
	Journal = {Constructive Approximation},
	Number = {2},
	Pages = {853--899},
	Title = {An Embedding of {R}e{LU} Networks and an Analysis of Their Identifiability},
	Ty = {JOUR},
	Url = {https://doi.org/10.1007/s00365-022-09578-1},
	Volume = {57},
	Year = {2023}}

@article{VLACIC2021107485,
title = {Affine symmetries and neural network identifiability},
journal = {Advances in Mathematics},
volume = {376},
pages = {107485},
year = {2021},
issn = {0001-8708},
doi = {https://doi.org/10.1016/j.aim.2020.107485},
url = {https://www.sciencedirect.com/science/article/pii/S0001870820305132},
author = {Verner Vlačić and Helmut Bölcskei}
}

@inproceedings{jacot2018neural,
 author = {Jacot, Arthur and Gabriel, Franck and Hongler, Clement},
 booktitle = {Advances in Neural Information Processing Systems},
 noeditor = {S. Bengio and H. Wallach and H. Larochelle and K. Grauman and N. Cesa-Bianchi and R. Garnett},
 pages = {},
 publisher = {Curran Associates, Inc.},
 title = {Neural Tangent Kernel: Convergence and Generalization in Neural Networks},
 url = {https://proceedings.neurips.cc/paper_files/paper/2018/file/5a4be1fa34e62bb8a6ec6b91d2462f5a-Paper.pdf},
 volume = {31},
 year = {2018}
}

@inproceedings{banerjee2023neural,
title={Neural Tangent Kernel at Initialization: Linear Width Suffices},
author={Arindam Banerjee and Pedro Cisneros-Velarde and Libin Zhu and Mikhail Belkin},
booktitle={The 39th Conference on Uncertainty in Artificial Intelligence},
year={2023},
url={https://openreview.net/forum?id=VJaoe7Rp9tZ}
}

@inproceedings{bombari2022memorization,
 author = {Bombari, Simone and Amani, Mohammad Hossein and Mondelli, Marco},
 booktitle = {Advances in Neural Information Processing Systems},
 noeditor = {S. Koyejo and S. Mohamed and A. Agarwal and D. Belgrave and K. Cho and A. Oh},
 pages = {7628--7640},
 publisher = {Curran Associates, Inc.},
 title = {Memorization and Optimization in Deep Neural Networks with Minimum Over-parameterization},
 url = {https://proceedings.neurips.cc/paper_files/paper/2022/file/323746f0ae2fbd8b6f500dc2d5c5f898-Paper-Conference.pdf},
 volume = {35},
 year = {2022}
}

@article{montanari2022interpolation,
  title={The interpolation phase transition in neural networks: Memorization and generalization under lazy training},
  author={Montanari, Andrea and Zhong, Yiqiao},
  journal={The Annals of Statistics},
  volume={50},
  number={5},
  pages={2816--2847},
  year={2022},
  publisher={Institute of Mathematical Statistics},
  URL = {https://doi.org/10.1214/22-AOS2211}
}

@misc{charisopoulos2019tropical,
      title={A Tropical Approach to Neural Networks with Piecewise Linear Activations}, 
      author={Vasileios Charisopoulos and Petros Maragos},
      year={2019},
      eprint={1805.08749},
      archivePrefix={arXiv},
      primaryClass={stat.ML},
      url={https://arxiv.org/abs/1805.08749}, 
}

@inproceedings{liang2025implicit,
title={Implicit Bias of Mirror Flow for Shallow Neural Networks in Univariate Regression},
author={Shuang Liang and Guido Mont\'ufar},
booktitle={The Thirteenth International Conference on Learning Representations},
year={2025},
url={https://openreview.net/forum?id=IF0Q9KY3p2}
}

@article{JMLR:v24:21-0832,
  author  = {Hui Jin and Guido Mont\'ufar},
  title   = {Implicit Bias of Gradient Descent for Mean Squared Error Regression with Two-Layer Wide Neural Networks},
  journal = {Journal of Machine Learning Research},
  year    = {2023},
  volume  = {24},
  number  = {137},
  pages   = {1--97},
  url     = {http://jmlr.org/papers/v24/21-0832.html}
}

@inproceedings{haase2023lower,
title={Lower Bounds on the Depth of Integral {R}e{LU} Neural Networks via Lattice Polytopes},
author={Christian Alexander Haase and Christoph Hertrich and Georg Loho},
booktitle={The Eleventh International Conference on Learning Representations },
year={2023},
url={https://openreview.net/forum?id=2mvALOAWaxY}
}

@inproceedings{NEURIPS2021_1b9812b9,
 author = {Hertrich, Christoph and Basu, Amitabh and Di Summa, Marco and Skutella, Martin},
 booktitle = {Advances in Neural Information Processing Systems},
 noeditor = {M. Ranzato and A. Beygelzimer and Y. Dauphin and P.S. Liang and J. Wortman Vaughan},
 pages = {3336--3348},
 publisher = {Curran Associates, Inc.},
 title = {Towards Lower Bounds on the Depth of {R}e{LU} Neural Networks},
 url = {https://proceedings.neurips.cc/paper_files/paper/2021/file/1b9812b99fe2672af746cefda86be5f9-Paper.pdf},
 volume = {34},
 year = {2021}
}

@inproceedings{arora2018understanding,
title={Understanding Deep Neural Networks with Rectified Linear Units},
author={Raman Arora and Amitabh Basu and Poorya Mianjy and Anirbit Mukherjee},
booktitle={International Conference on Learning Representations},
year={2018},
url={https://openreview.net/forum?id=B1J_rgWRW},
}

@inproceedings{matena2022a,
title={A Combinatorial Perspective on the Optimization of Shallow {R}e{LU} Networks},
author={Michael S Matena and Colin Raffel},
booktitle={Advances in Neural Information Processing Systems},
noeditor={Alice H. Oh and Alekh Agarwal and Danielle Belgrave and Kyunghyun Cho},
year={2022},
url={https://openreview.net/forum?id=GbpEszOdiTV}
}

@article{doi:10.1137/23M1565504,
author = {Kohn, Kathl\'{e}n and Mont\'{u}far, Guido and Shahverdi, Vahid and Trager, Matthew},
title = {Function Space and Critical Points of Linear Convolutional Networks},
journal = {SIAM Journal on Applied Algebra and Geometry},
volume = {8},
number = {2},
pages = {333-362},
year = {2024},
doi = {10.1137/23M1565504},
URL = {https://doi.org/10.1137/23M1565504},
eprint = {https://doi.org/10.1137/23M1565504
 }
}

@misc{escobar2023enumeration,
      title={Enumeration of max-pooling responses with generalized permutohedra}, 
      author={Laura Escobar and Patricio Gallardo and Javier González-Anaya and José L. González and Guido Montúfar and Alejandro H. Morales},
      year={2023},
      eprint={2209.14978},
      archivePrefix={arXiv},
      primaryClass={math.CO},
      url={https://arxiv.org/abs/2209.14978}, 
}

@inproceedings{humayun2024deep,
title={Deep Networks Always Grok and Here is Why},
author={Ahmed Imtiaz Humayun and Randall Balestriero and Richard Baraniuk},
booktitle={High-dimensional Learning Dynamics 2024: The Emergence of Structure and Reasoning},
year={2024},
url={https://openreview.net/forum?id=NpufNsg1FP}
}

@article{brandenburg2024the,
title={The Real Tropical Geometry of Neural Networks for Binary Classification},
author={Marie-Charlotte Brandenburg and Georg Loho and Guido Mont\'ufar},
journal={Transactions on Machine Learning Research},
issn={2835-8856},
year={2024},
url={https://openreview.net/forum?id=I7JWf8XA2w},
note={}
}

@inproceedings{patel2025on,
title={On the Local Complexity of Linear Regions in Deep {R}e{LU} Networks},
author={Niket Nikul Patel and Guido Mont\'ufar},
year={2025},
url={https://openreview.net/forum?id=IQdlPvj4dX}, 
booktitle = 	 {Proceedings of the International Conference on Machine Learning},
note={To appear}
}

@InProceedings{pmlr-v80-serra18b,
  title = 	 {Bounding and Counting Linear Regions of Deep Neural Networks},
  author =       {Serra, Thiago and Tjandraatmadja, Christian and Ramalingam, Srikumar},
  booktitle = 	 {Proceedings of the 35th International Conference on Machine Learning},
  pages = 	 {4558--4566},
  year = 	 {2018},
  noeditor = 	 {Dy, Jennifer and Krause, Andreas},
  volume = 	 {80},
  series = 	 {Proceedings of Machine Learning Research},
  month = 	 {10--15 Jul},
  publisher =    {PMLR},
  url = 	 {https://proceedings.mlr.press/v80/serra18b.html}
}

@article{noteson,
      title={Notes on the number of linear regions of deep neural networks}, 
      author={Guido Mont\'ufar},
      year={2017},
      journal={Presented at SampTA 2017 – Sampling Theory and Applications},
      url={https://www.researchgate.net/publication/322539221_Notes_on_the_number_of_linear_regions_of_deep_neural_networks}
}

@article{Lim2022best,
	Author = {Lim, Lek-Heng and Micha{\l}ek, Mateusz and Qi, Yang},
	Da = {2022/02/01},
	noDoi = {10.1007/s00365-021-09545-2},
	Isbn = {1432-0940},
	Journal = {Constructive Approximation},
	Number = {1},
	Pages = {583--604},
	Title = {Best k-Layer Neural Network Approximations},
	Ty = {JOUR},
	Url = {https://doi.org/10.1007/s00365-021-09545-2},
	Volume = {55},
	Year = {2022},
    }

@InProceedings{10.1007/978-3-319-97798-0_4,
author="Mont{\'u}far, Guido",
noeditor="Ay, Nihat
and Gibilisco, Paolo
and Mat{\'u}{\v{s}}, Franti{\v{s}}ek ",
title="Restricted {B}oltzmann Machines: Introduction and Review",
booktitle="Information Geometry and Its Applications ",
year="2018",
publisher="Springer International Publishing",
address="Cham",
pages="75--115",
noisbn="978-3-319-97798-0"
}

@article{karhadkar2024mildly,
title={Mildly Overparameterized {R}e{LU} Networks Have a Favorable Loss Landscape},
author={Kedar Karhadkar and Michael Murray and Hanna Tseran and Guido Mont\'ufar},
journal={Transactions on Machine Learning Research},
issn={2835-8856},
year={2024},
url={https://openreview.net/forum?id=10WARaIwFn},
note={}
}

@inproceedings{karhadkar2024bounds,
title={Bounds for the smallest eigenvalue of the {NTK} for arbitrary spherical data of arbitrary dimension},
author={Kedar Karhadkar and Michael Murray and Guido Mont\'ufar},
booktitle={The Thirty-eighth Annual Conference on Neural Information Processing Systems},
year={2024},
url={https://openreview.net/forum?id=mHVmsy9len}
}

@article{doi:10.1137/21M1413699,
author = {Mont\'{u}far, Guido and Ren, Yue and Zhang, Leon},
title = {Sharp Bounds for the Number of Regions of Maxout Networks and Vertices of {M}inkowski Sums},
journal = {SIAM Journal on Applied Algebra and Geometry},
volume = {6},
number = {4},
pages = {618-649},
year = {2022},
nodoi = {10.1137/21M1413699},
URL = {https://doi.org/10.1137/21M1413699},
eprint = {https://doi.org/10.1137/21M1413699}
    }

@InProceedings{zhang2018tropical,
  title = 	 {Tropical Geometry of Deep Neural Networks},
  author = 	 {Zhang, Liwen and Naitzat, Gregory and Lim, Lek-Heng},
  booktitle = 	 {Proceedings of the 35th International Conference on Machine Learning},
  pages = 	 {5824--5832},
  year = 	 {2018},
  noeditor = 	 {Dy, Jennifer and Krause, Andreas},
  volume = 	 {80},
  series = 	 {Proceedings of Machine Learning Research},
  noaddress = 	 {Stockholmsmässan, Stockholm Sweden},
  month = 	 {10--15 Jul},
  publisher = 	 {PMLR},
  url = 	 {http://proceedings.mlr.press/v80/zhang18i.html}
}

@inproceedings{pascanu2014number,
      title={On the number of response regions of deep feed forward networks with piece-wise linear activations}, 
      author={Razvan Pascanu and Guido Mont\'ufar and Yoshua Bengio},
      booktitle ={International Conference on Learning Representations}, 
      year={2014},
      url={https://openreview.net/forum?id=bSaT4mmQt84Lx}, 
}

@inproceedings{NEURIPS2019_9766527f,
 author = {Hanin, Boris and Rolnick, David},
 booktitle = {Advances in Neural Information Processing Systems},
 noeditor = {H. Wallach and H. Larochelle and A. Beygelzimer and F. d\textquotesingle Alch\'{e}-Buc and E. Fox and R. Garnett},
 pages = {},
 publisher = {Curran Associates, Inc.},
 title = {Deep {R}e{LU} Networks Have Surprisingly Few Activation Patterns},
 url = {https://proceedings.neurips.cc/paper_files/paper/2019/file/9766527f2b5d3e95d4a733fcfb77bd7e-Paper.pdf},
 volume = {32},
 year = {2019}
}

@inproceedings{NEURIPS2021_f2c3b258,
 author = {Tseran, Hanna and Mont\'ufar, Guido},
 booktitle = {Advances in Neural Information Processing Systems},
 noeditor = {M. Ranzato and A. Beygelzimer and Y. Dauphin and P.S. Liang and J. Wortman Vaughan},
 pages = {28995--29008},
 publisher = {Curran Associates, Inc.},
 title = {On the Expected Complexity of Maxout Networks},
 url = {https://proceedings.neurips.cc/paper_files/paper/2021/file/f2c3b258e9cd8ba16e18f319b3c88c66-Paper.pdf},
 volume = {34},
 year = {2021}
}

@inproceedings{NIPS2014_109d2dd3,
 author = {Mont\'ufar, Guido and Pascanu, Razvan and Cho, Kyunghyun and Bengio, Yoshua},
 booktitle = {Advances in Neural Information Processing Systems},
 noeditor = {Z. Ghahramani and M. Welling and C. Cortes and N. Lawrence and K.Q. Weinberger},
 pages = {},
 publisher = {Curran Associates, Inc.},
 title = {On the Number of Linear Regions of Deep Neural Networks},
 url = {https://proceedings.neurips.cc/paper_files/paper/2014/file/109d2dd3608f669ca17920c511c2a41e-Paper.pdf},
 volume = {27},
 year = {2014}
}

@article{kohn2022geometry,
author = {Kohn, Kathl\'{e}n and Merkh, Thomas and Mont\'{u}far, Guido and Trager, Matthew},
title = {Geometry of Linear Convolutional Networks},
journal = {SIAM Journal on Applied Algebra and Geometry},
volume = {6},
number = {3},
pages = {368-406},
year = {2022},
nodoi = {10.1137/21M1441183},
URL = {https://doi.org/10.1137/21M1441183},
eprint = {https://doi.org/10.1137/21M1441183}
}

@article{grigsby2022functional,
      title={Functional dimension of feedforward {R}e{LU} neural networks}, 
      author={J. Elisenda Grigsby and Kathryn Lindsey and Robert Meyerhoff and Chenxi Wu},
      year={2022},
      journal={arXiv preprint arXiv:2209.04036},
      archivePrefix={arXiv},
      primaryClass={math.MG}
}

@InProceedings{pmlr-v119-rolnick20a,
  title = 	 {Reverse-engineering deep {R}e{LU} networks},
  author =       {Rolnick, David and Kording, Konrad},
  booktitle = 	 {Proceedings of the 37th International Conference on Machine Learning},
  pages = 	 {8178--8187},
  year = 	 {2020},
  noeditor = 	 {III, Hal Daumé and Singh, Aarti},
  volume = 	 {119},
  series = 	 {Proceedings of Machine Learning Research},
  month = 	 {13--18 Jul},
  publisher =    {PMLR},
  url = 	 {https://proceedings.mlr.press/v119/rolnick20a.html}
}

@inproceedings{phuong2020functional,
title={Functional vs. parametric equivalence of {R}e{LU} networks},
author={Mary Phuong and Christoph H. Lampert},
booktitle={International Conference on Learning Representations},
year={2020},
url={https://openreview.net/forum?id=Bylx-TNKvH}
}

@InProceedings{pmlr-v80-laurent18b,
  title = 	 {The Multilinear Structure of {R}e{LU} Networks},
  author =       {Laurent, Thomas and von Brecht, James},
  booktitle = 	 {Proceedings of the 35th International Conference on Machine Learning},
  pages = 	 {2908--2916},
  year = 	 {2018},
  noeditor = 	 {Dy, Jennifer and Krause, Andreas},
  volume = 	 {80},
  series = 	 {Proceedings of Machine Learning Research},
  month = 	 {10--15 Jul},
  publisher =    {PMLR},
  url = 	 {https://proceedings.mlr.press/v80/laurent18b.html}
}

@inproceedings{nguyen2021tight,
  title={Tight bounds on the smallest eigenvalue of the neural tangent kernel for deep {R}e{LU} networks},
  author={Nguyen, Quynh and Mondelli, Marco and Mont\'ufar, Guido},
  booktitle={International Conference on Machine Learning},
  pages={8119--8129},
  year={2021},
  organization={PMLR}
}

@book {MichalekSturmfels2019InvitationNonlinearAlgebraTEXT,
    AUTHOR = {Micha\l{}ek, Mateusz and Sturmfels, Bernd},
    TITLE = {Invitation to Nonlinear Algebra},
    SERIES = {Graduate Studies in Mathematics},
    PUBLISHER = {American Mathematical Society, Providence, RI},
    YEAR = {2021}
}

@Misc{M2,author = {Grayson, Daniel R. and Stillman, Michael E.},
title = {Macaulay2, a software system for research in algebraic geometry},
howpublished = {Available at \url{http://www2.macaulay2.com}}
}

@book{sullivant2023algebraic,
  title={Algebraic statistics},
  author={Sullivant, Seth},
  volume={194},
  year={2023},
  publisher={American Mathematical Society}
}

@inproceedings{pokutta,
title={Approximating Latent Manifolds in Neural Networks via Vanishing Ideals},
author={Nico Pelleriti and Max Zimmer and Elias Samuel Wirth and Sebastian Pokutta},
booktitle={Forty-second International Conference on Machine Learning},
year={2025},
url={https://openreview.net/forum?id=WYlerYGDPL}
}

@book{cox2015ideals,
author = {Cox, David A. and Little, John and O'Shea, Donal},
title = {Ideals, Varieties, and Algorithms: An Introduction to Computational Algebraic Geometry and Commutative Algebra},
year = {2010},
isbn = {1441922571},
publisher = {Springer Publishing Company, Incorporated},
edition = {3rd},
}

@article{grunwald2019mdl,
author = {Gr\"{u}nwald, Peter and Roos, Teemu},
title = {Minimum description length revisited},
journal = {International Journal of Mathematics for Industry},
volume = {11},
number = {01},
pages = {1930001},
year = {2019},
doi = {10.1142/S2661335219300018},

URL = { 
    
        https://doi.org/10.1142/S2661335219300018
    
    

},
eprint = { 
    
        https://doi.org/10.1142/S2661335219300018
    
    

}

}
\end{document}